\documentclass[11pt,reqno,oneside]{amsart}

\usepackage{graphicx}
\usepackage{amssymb}
\usepackage{epstopdf}
\usepackage{cases}
\usepackage[abbrev,nobysame]{amsrefs} \AtBeginDocument{\def\MR#1{}} 
\usepackage{hyperref}

\usepackage[a4paper, total={6in, 9in}]{geometry}
\usepackage{amsmath,amsfonts,amsthm,mathrsfs,amssymb,cite}
\usepackage[usenames]{color}

\newtheorem{thm}{Theorem}[section]

\newtheorem{lem}{Lemma}[section]
\newtheorem{prop}{Proposition}[section]
\theoremstyle{definition}

\theoremstyle{remark}

\newtheorem{rem}{Remark}[section]
\numberwithin{equation}{section}

\numberwithin{equation}{section}

\allowdisplaybreaks

\newcommand{\mR}{\mathbb{R}}                    
\newcommand{\abs}[1]{\lvert #1 \rvert}          
\newcommand{\norm}[1]{\lVert #1 \rVert}         

\newcommand{\ol}[1]{\overline{#1}}

\newcommand{\p}{\partial}

\newcommand{\R}{\mathbb{R}} \newcommand{\mathR}{\mathbb{R}}
\newcommand{\Rn}{{\mathR^n}}

\newcommand{\supp}{\mathop{\rm supp}}
\DeclareMathOperator{\divr}{div}

\newcommand{\dif}[1]{\,\mathrm{d}{#1}} 
\newcommand{\nrm}[2][]{ \| {#2} \|_{#1}}
\newcommand{\agl}[1][\cdot]{ \langle {#1} \rangle}

\usepackage{tikz}

\newcommand{\df}{\mathrm{d}}
\newcommand{\calX}{{\mathcal{X}}}

\newcommand{\comment}[1]{} 		

\newcommand{\sq}[1]{{#1}}

\title[Fixed angle inverse scattering]{Fixed angle inverse scattering in the presence of a Riemannian metric}

\author{Shiqi Ma}
\address{Department of Mathematics and Statistics, Universify of Jyv\"askyl\"a, Jyv\"askyl\"a, Finland}
\email{mashiqi01@gmail.com, shiqi.s.ma@jyu.fi}

\author{Mikko Salo}
\address{Department of Mathematics and Statistics, Universify of Jyv\"askyl\"a, Jyv\"askyl\"a, Finland}
\email{mikko.j.salo@jyu.fi}

\dedicatory{Dedicated to Michael Klibanov on the occasion of his 70th birthday}

\begin{document}

\begin{abstract}
	
	We consider a fixed angle inverse scattering problem in the presence of a known Riemannian metric. First, assuming a no caustics condition, we study the direct problem by utilizing the progressing wave expansion. Under a symmetry assumption on the metric, we obtain uniqueness and stability results in the inverse scattering problem for a potential with data generated by two incident waves from opposite directions. Further, similar results are given using one measurement provided the potential also satisfies a symmetry assumption. This work extends the results of \cite{RSfixed2019March, RSfixed2019} from the Euclidean case to certain Riemannian metrics.
	
	\medskip
	
	
	
\end{abstract}

\maketitle

\section{Introduction} \label{sec:intro-ACR2019}

We study the fixed angle inverse scattering problem in the presence of a Riemannian metric. In Euclidean space, this problem corresponds to determining a potential $q \in C^{\infty}_c(\R^n)$ from the knowledge of the scattering amplitude $a_q(\,\cdot\,,\omega,\,\,\cdot\,)$ for a fixed direction $\omega \in S^{n-1}$. 
\sq{As discussed in \cite{RSfixed2019}, an equivalent inverse problem is to} determine $q$ by measuring $U|_{\partial B \times \R}$ where $U$ solves the wave equation
\begin{equation} \label{eq:0-ACR2019}
\left\{\begin{aligned}
	& (\partial_t^2 - \Delta + q)U(x,t) = 0 \text{ in $\R^{n+1}$} \\
	& U|_{\{ t \ll 0 \}} = \delta(t-x\cdot \omega),
\end{aligned}\right.
\end{equation}
and where $B$ is a ball containing the support of $q$.

In the recent articles \cite{RSfixed2019March} and \cite{RSfixed2019}, it was proved that measurements corresponding to \emph{two} fixed directions $\pm \omega \in S^{n-1}$ uniquely determine a potential $q \in C^{\infty}_c(\R^n;\R)$. 
Here $C^{\infty}_c(\R^n;\R)$ stands for the set of real-valued $C^\infty$-smooth functions with compact supports.
The objective of this work is to investigate to what extent it is possible to replace the Euclidean metric on $\R^n$ by a known Riemannian metric. 
This corresponds to determining a potential $q$ from two scattering measurements in the presence of a known nonconstant sound speed.

%
We establish some notations in order to state the main theorems. Denote by $B$ the open unit ball in $\R^n$. 
Let $x$ be Cartesian coordinates in $\R^n$, and let $g$ be a smooth Riemannian metric on $\R^n$ so that 
\begin{equation} \label{asm:gce-ACR2019}
g_{jk}(x) = \delta_{jk} \text{ for } |x| \geq 1.
\end{equation}
Let $\Delta_g$ be the Laplace-Beltrami operator, given in coordinates by 
\[
\Delta_g u = \abs{g}^{-1/2} \p_j(\abs{g}^{1/2} g^{jk} \p_k u)
\]
where $(g^{jk}) = (g_{jk})^{-1}$ and $\abs{g} = \det(g_{jk})$. Here and throughout this article, we use the Einstein summation convention where a repeated index in upper and lower position is summed.

Given $q \in C^{\infty}_c(\R^n)$ supported in $\overline{B}$, let $U_q^{\pm} = U_q^{\pm}(x,t)$ be the unique solution of
\begin{equation} \label{eq:1-ACR2019}
	\left\{\begin{aligned}
& (\partial_t^2 - \Delta_g + q)U_q^{\pm} = 0 \text{ in $\R^{n+1}$}, \\
& U_q^{\pm}|_{\{ t < -1\}} = \delta(t \mp x_n).
\end{aligned}\right.
\end{equation}
The assumption \eqref{asm:gce-ACR2019} ensures that $\delta(t \mp x_n)$ solves the free equation $(\partial_t^2 - \Delta_g) U = 0$ for $t < -1$. Then $U_q^{\pm}$ is the solution related to the incoming plane wave $\delta(t \mp x_n)$ propagating in direction $\pm e_n$. We will make an \emph{absence of caustics} assumption (here $B_{\pm} = B \cap \{ \pm x_n \geq 0 \}$): 
\[
\text{There are $\omega_{\pm} \in C^{\infty}(\ol{B})$ with $|\df \omega_\pm|_g = 1$ in $B$, $\omega_{\pm} = x_n$ to first order on $\p B_{\pm}$.}
\]
\sq{Here $|\df \omega_\pm|_g := \sqrt{g^{jk} (\partial_{x_j} \omega) (\partial_{x_k} \omega)}$ is the $g$-norm of $\df \omega_{\pm}$.}
This assumption implies that the solutions $U_q^{\pm}$ have the following explicit representations in $\ol{B} \times \mR$ (see Lemma \ref{lem:dp-ACR2019}): one has 
\[
U_q^{\pm}(x,t) = \delta(t \mp \omega_{\pm}(x)) + u_q^{\pm}(x,t) H(t \mp \omega_{\pm}(x))
\]
where $H(t)$ is the Heaviside function, and $u_q^{\pm}$ is smooth in $\{ t \geq \pm \omega_{\pm}(x) \}$ and vanishes in $\{ t < \pm \omega_{\pm}(x) \}$. Thus measuring $U_q^{\pm}|_{\p B \times \mR}$ corresponds to measuring $u_q^{\pm}|_{\p B \times \mR}$, and the latter quantity is well defined pointwise.

We wish to give an analogue of the results in \cite{RSfixed2019March} and \cite{RSfixed2019} stating that measurements from two fixed directions $\pm e_n$ uniquely determine the potential $q$. The method relies on putting the solutions $u_q^+$ in $\{ t \geq \omega_+(x) \}$ and $u_q^-$ in $\{ t \geq -\omega_-(x) \}$ together by applying the reflection $t \mapsto -t$ to the latter solution. For this, we need the interfaces $\{ t = \omega_+(x) \}$ and $\{ t = \omega_-(x) \}$ to match, i.e.\ we need $\omega_+ = \omega_-$. This is guaranteed by the following condition:
\begin{equation} \label{assumption1-ACR2019}
\left\{ \begin{array}{c} \text{There is $\omega \in C^{\infty}(\ol{B})$ with $\abs{d\omega}_g = 1$ in $B$ } \\[3pt]
\text{and $\omega = x_n$ to infinite order on $\p B$.} \end{array} \right.
\end{equation}
This condition combines an absence of caustics assumption and a symmetry assumption for the metric $g$. In fact, by Lemma \ref{lemma_omega_equivalence} each of the following two conditions is equivalent with \eqref{assumption1-ACR2019}:
\begin{align}
 & \left\{ \begin{array}{l} \text{$\R^n$ is smoothly parametrized by $g$-geodesics starting from $\{ x_n = -1 \}$ in} \\[3pt]
\text{direction $e_n$, and any such geodesic meets $\{ x_n = 1 \}$ pointing in direction $e_n$.} \end{array} \right.  \label{assumption1a-ACR2019} \tag{1.4a} \\[3pt]
 & \left\{ \begin{array}{l} \text{There is a global coordinate system $y = (y',y_n)$ in $\R^n$ with} \\[3pt]
\text{$\partial_{y_n} = \partial_{x_n}$ outside $B$, so that $g(y) = \left( \begin{array}{cc} g_0(y) & 0 \\ 0 & 1 \end{array} \right)$.} \end{array} \right. \label{assumption1b-ACR2019}  \tag{1.4b}
\end{align}

Finally, the method is based on Carleman estimates for the wave equation, which are available in particular under the following assumption (see the discussion after Lemma \ref{lem:CarlEst-ACR2019} as well as Lemma \ref{lem:weight-ACR2019}):
\begin{equation} \label{ass2-ACR2019}
\text{$(\overline{B}, g)$ admits a smooth strictly convex function with no critical point in $\overline B$.}
\end{equation}
Our first main theorem states that under these assumptions it is possible to determine an arbitrary potential $q$ from two measurements.

\begin{thm} \label{thm:mst-ACR2019}
Let $g$ be a smooth Riemannian metric on $\R^n$, $n \geq 2$, satisfying \eqref{asm:gce-ACR2019}, \eqref{assumption1-ACR2019} and \eqref{ass2-ACR2019}. There is $T > 0$ such that for any $q_1, q_2 \in C^{\infty}_c(\R^n;\R)$ supported in $\overline{B}$, if one has 
\[
u_{q_1}^{\pm}|_{\partial B \times [-1,T]} = u_{q_2}^{\pm}|_{\partial B \times [-1,T]},
\]
then $q_1 = q_2$.
\end{thm}

It follows from \eqref{assumption1b-ACR2019} that, for example, any metric of the form 
\[
g(x) = \left( \begin{array}{cc} I_{n-1} + h(x) & 0 \\ 0 & 1 \end{array} \right),
\]
where $h \in C^{\infty}_c(\mR^n, \mR^{(n-1) \times (n-1)})$ is supported in $\overline{B}$ and $\norm{h}_{C^2}$ is sufficiently small, satisfies \eqref{asm:gce-ACR2019}, \eqref{assumption1-ACR2019} and \eqref{ass2-ACR2019}.

Moreover, if $(y',y_n)$ are as in \eqref{assumption1b-ACR2019} and if $g$ and $q$ satisfy the symmetry properties 
\begin{align} 
g(y',-y_n) &= g(y',y_n), \label{assumptiong-ACR2019} \\
 q(y',-y_n) &= q(y',y_n), \label{assumptionq-ACR2019}
\end{align}
then $q$ is uniquely determined by a single measurement $u_q^+|_{\partial B \times [-1,T]}$. This follows immediately from Theorem \ref{thm:mst-ACR2019} since $u_q^-(y',y_n,t) = u_q^+(y',-y_n,t)$ (it is enough to note that $U = U_q^+(y',-y_n,t)$ solves $(\partial_t^2 - \Delta_g + q) U = 0$ in $\R^{n+1}$ with $U = \delta(t+x_n)$ for $t < -1$, using \eqref{assumption1b-ACR2019}, \eqref{assumptiong-ACR2019} and \eqref{assumptionq-ACR2019}).

\begin{thm} \label{thm:One1-ACR2019}
Let $g$ be a smooth Riemannian metric on $\R^n$, $n \geq 2$, satisfying \eqref{asm:gce-ACR2019}, \eqref{assumption1-ACR2019},  \eqref{ass2-ACR2019} and \eqref{assumptiong-ACR2019}.
There is $T > 0$ such that for any $q_1, q_2 \in C^{\infty}_c(\R^n;\R)$ supported in $\overline{B}$ and satisfying \eqref{assumptionq-ACR2019}, if one has 
\[
u_{q_1}^{+}|_{\partial B \times [-1,T]} = u_{q_2}^{+}|_{\partial B \times [-1,T]},
\]
then $q_1 = q_2$.
\end{thm}

Theorems \ref{thm:mst-ACR2019} and \ref{thm:One1-ACR2019} come with Lipschitz stability estimates (see Section \ref{sec:ip2-ACR2019} for the precise statement). As in \cite{RSfixed2019} it is also possible to prove uniqueness and stability results in the single measurement case when the odd part of $q_1 - q_2$ with respect to the reflection $(y',y_n) \mapsto (y',-y_n)$ is small compared to $q_1 - q_2$. We omit the details.

\subsection{Discussion and connection to the existing results} \label{subsec:DiscConn-ACR2019}

The fixed angle inverse scattering problem is a formally determined nonlinear inverse problem in any dimension $n \geq 1$. 
It can be studied either in the frequency domain, as the problem of determining $q$ from the scattering amplitude $a_q(\,\cdot\,,\omega,\,\cdot\,)$ for the Schr\"odinger operator $-\Delta + q$ with a fixed direction $\omega \in S^{n-1}$, or as the problem of recovering $q$ from boundary or scattering measurements of the solution $U_q$ of the wave equation. 
The equivalence of these two problems is proved in detail in \cite{RSfixed2019} (see also \cite{Melrose,  Uhlmann_backscattering,  MelroseUhlmann_bookdraft} for the case of odd dimensions).

The one-dimensional case is classical and has been thoroughly studied, see \cite{Marchenko, DeiftTrubowitz}. 
When $n \geq 2$ there are results on uniqueness for small or generic potentials \cite{Stefanov_generic, BarceloEtAl}, recovery of  singularities \cite{Ruiz, Meronno_thesis}, and recovery of the zero potential \cite{BaylissLiMorawetz}. 
Recently it was proved in \cite{RSfixed2019March, RSfixed2019} that measurements corresponding to two opposite fixed angles uniquely determine a general smooth compactly supported potential. 
The problem with one measurement is still open in general, but uniqueness was proved in \cite{RSfixed2019March, RSfixed2019} for potentials satisfying certain symmetry conditions or horizontal control conditions (analogous to angularly controlled potentials in backscattering, see \cite{RakeshUhlmann}). 
One also has Lipschitz stability estimates for the wave equation version of the problem. 
We also mention the work \cite{MeronnoPotencianoSalo} which studies the fixed angle problem with first order coefficients.

All the above results consider the Euclidean metric. In this paper we extend the approach of \cite{RSfixed2019March, RSfixed2019} to Riemannian metrics satisfying the conditions \eqref{asm:gce-ACR2019}, \eqref{assumption1-ACR2019} and \eqref{ass2-ACR2019}. 
The argument follows the Bukhgeim-Klibanov method, see \cite{BukhgeimKlibanov} and \cite{IY01}, on the use of Carleman estimates for formally determined inverse problems. 
Further information about this method and its variants may be found in  \cite{Kh89, Bu00, Be04, Is06, LCW2009, Kl13, SU13, BY17}. We also refer to \cite{Uh00} for inverse scattering problems and to \cite{BY17} for the Bukhgeim-Klibanov method for Riemannian metrics.

We now briefly describe the argument. 
The idea is that if $u_{q_1}|_{\p B \times [-1,T]} = u_{q_2}|_{\p B \times [-1,T]}$, then $\tilde u = u_{q_1} - u_{q_2}$ solves the wave equation 
\begin{equation} \label{wave_eq_subtracted}
(\p_t^2 - \Delta_g + q_1) \tilde u = -(q_1-q_2)(x) u_2(x,t)
\end{equation}
in a modified space-time cylinder $Q_+$ with zero lateral boundary values. By looking at the exterior problem, also $\p_{\nu} \tilde u$ vanishes at the lateral boundary.

Now if the right hand side of \eqref{wave_eq_subtracted} were zero, one could use the unique continuation principle for the wave equation (proved by using Carleman estimates) to conclude that $\tilde u=0$ and hence $q_1 = q_2$. 
However, the right hand side source is not zero, but it has the special form $-\tilde q(x) u_2(x,t)$ where $\tilde q = q_1 - q_2$. This source is of the same form as in the Bukhgeim-Klibanov method. 
We then follow this method and apply a Carleman estimate with pseudoconvex weight that is large at the bottom of the cylinder $Q_+$, which effectively converts the source term in \eqref{wave_eq_subtracted} into a boundary term at the bottom of $Q_+$. 
(In effect, this means that we still treat the problem as some kind of unique continuation problem, even if the right hand side source does not vanish.) 
The final issue is that the Carleman estimate has other boundary terms at the bottom of $Q_+$ with unfavorable signs. 
To deal with this, we use a second measurement from direction $-e_n$ and do a reflection argument which will cancel the boundary terms and allow us to conclude that $\tilde u =0$ and thus $q_1 = q_2$.

In order to implement the above approach in the Riemannian setting, we need the metric to satisfy the three conditions \eqref{asm:gce-ACR2019}, \eqref{assumption1-ACR2019} and \eqref{ass2-ACR2019}. 
We now describe the role of these assumptions in more detail.
\begin{itemize}
\item 
The condition \eqref{asm:gce-ACR2019} guarantees that the plane waves $\delta(t\mp x_n)$ solve the wave equation for $t \ll 0$, and that geodesics outside $B$ are straight lines.
\item 
The assumption \eqref{assumption1-ACR2019} imposes both a symmetry condition and an absence of caustics condition on the metric, which may be seen from the equivalent conditions \eqref{assumption1a-ACR2019} and \eqref{assumption1b-ACR2019}. The symmetry is required for the reflection argument mentioned above.
\item 
The assumption \eqref{ass2-ACR2019} on the existence of a strictly convex function in $(\ol{B}, g)$ ensures that Carleman estimates with pseudoconvex weights are available for the wave equation. This assumption is discussed in \cite[Lemma 2.1]{PaternainSaloUhlmannZhou} when $(\overline{B},g)$ has strictly convex boundary: in particular \eqref{ass2-ACR2019} is valid when the sectional curvature does not change sign or there are no focal points. Clearly \eqref{ass2-ACR2019} is stable under small perturbations of the metric $g$.

\end{itemize}

This article is organized as follows. Section \ref{sec:intro-ACR2019} is the introduction and states the main results. 
Section \ref{sec:DeRie-ACR2019} introduces some notations and formulates the solution of the direct problem.
In Section \ref{sec:energy-ACR2019}, certain energy estimates are given.
\sq{Certain necessary Carleman estimates are treated in Section \ref{sec:CarlEst-ACR2019}.}
Section \ref{sec:ip2-ACR2019} is devoted to the recovery of the potential using two measurements. The equivalence of the conditions \eqref{assumption1-ACR2019}, \eqref{assumption1a-ACR2019} and \eqref{assumption1b-ACR2019} is proved in Appendix \ref{sec_equivalence}.

\subsection*{Acknowledgements}
\sq{We would like to thank the anonymous referee for several helpful comments that have improved the presentation.} Both authors were supported by the Academy of Finland (Finnish Centre of Excellence in Inverse Modelling and Imaging, grants 284715 and 309963). M.S.\ was also supported by ERC under Horizon 2020 (ERC CoG 770924).

\section{The analysis of the forward problem} \label{sec:DeRie-ACR2019}

In this section, we focus on the direct problem \eqref{eq:1-ACR2019}.
The existence of a solution is shown by the progressing wave expansion, and uniqueness is guaranteed by standard energy estimates.
We also present some basic properties of the solution which will be useful in solving the inverse problem.

\subsection{Notation and preliminaries} \label{subsec:PreRieGeo-ACR2019}

We adopt the conventional notation for the Dirac delta distribution $\delta$ and its derivative $\delta'$, i.e.~\(
(\delta,\varphi) = \varphi(0),
\)
and
\(
(\delta',\varphi) = -\varphi'(0).
\)
We denote the divergence operator as $\divr$, i.e.~\(
\divr (\alpha_k \df x^k) := |g|^{-1/2} \partial_j(|g|^{1/2} g^{jk} \alpha_k),
\)
and the Laplace-Beltrami operator $\Delta_g$ is defined as
\(
\Delta_g u := \divr (\df u) = \abs{g}^{-1/2} \p_j(\abs{g}^{1/2} g^{jk} \p_k u),
\)
as already mentioned in Section \ref{sec:intro-ACR2019}.
When $j \geq 0$, we define a one-dimensional distribution $H_j$ as
\begin{equation*}
	H_j(s) = 
	\begin{cases}
	s^j, & s \geq 0, \\
	0, & s < 0.
	\end{cases}
\end{equation*}
We also write $H_{-1}(s) := \delta(s)$, $H_{-2}(s) := \delta'(s)$ and $H_{-3}(s) := \delta''(s)$.
One should note that 
\begin{equation*}
	H_j'(s) = C_j H_{j-1}(s),
\end{equation*}
where $C_j = j$ when $j \geq 1$ and $C_j = 1$ when $j \leq 0$.

The $g$-inner product is denoted as $\agl[\df u, \df v] := g^{jk} (\partial_j u) (\partial_k v)$ and $\agl[\p_j, \p_k] |_p := g_{jk}(p)$.
The following equalities hold for sufficiently smooth functions:
\begin{subequations} \label{eq:rel14-ACR2019}
	\begin{numcases}{}
		\Delta_g (u v) 
		= \Delta_g u \cdot v + 2\agl[\df u, \df v] + u \cdot \Delta_g v,  \label{eq:relation5-ACR2019} \\
		\divr (u \!\dif v) 
		= u \cdot \Delta_g v + \agl[\df u, \df v],  \label{eq:relation1-ACR2019} \\
		\df \big( H_j(t - v(x)) \big)
		= -\df v H_j'(t-v), \label{eq:relation3-ACR2019} \\
		\Delta_g \big( H_j(t - v(x)) \big)
		= - (\Delta_g v) H_j'(t-v) + \agl[\df v, \df v] \ddot H_j(t-v). \label{eq:relation4-ACR2019}
	\end{numcases}
\end{subequations}

%
%
%
%
%
%
%
%
%
%
%
%

\subsection{Progressing wave expansion} \label{subsec:ExtSolRie-ACR2019}

Consider an incoming plane wave $\delta(t - x_n)$, and the corresponding solution of the wave equation 
\begin{equation} \label{eq:2-ACR2019}
\begin{cases}
(\partial_t^2 - \Delta_g + q)U &= 0 \text{ in $\R^{n+1}$}, \\
U|_{\{ t < -1\}} &= \delta(t - x_n).
\end{cases}
\end{equation}
If the underlying metric $g$ is Euclidean, the solution of this equation will have the form $U(x,t) = \delta(t - x_n) + u(x,t) H_0(t-x_n)$.
However, when $g$ is general Riemannian metric, the shape of the incident wave will be perturbed when passing through the support of $q$. We can represent the solution of \eqref{eq:2-ACR2019} by using the progressing wave expansion method (cf.~\cite{MelroseUhlmann_bookdraft}).

The progressing wave expansion method assumes that the solution has the form
\begin{equation} \label{eq:Uasm-ACR2019}
	U(x,t) = \sum_{j \geq -1}^N a_j(x, t) \cdot H_j (t - \omega(x)) + R_N(x,t), \quad (x,t) \in B \times \R,
\end{equation}
where $\omega(x) = x_n$ when $x_n < -1$, $a_{-1}(x,t) = 1$ when \sq{$x_n < -1$, $a_j(x,t) = 0$ when $x_n < -1$ for $j \geq 0$,} and $R_N$ is the remainder term which is more regular than the other terms. 
We only consider the representation \eqref{eq:Uasm-ACR2019} for $(x,t) \in B \times \R$.

\sq{In the progressing wave expansion one often assumes that the coefficients $a_j$ only depend on $x$, i.e.\ $a_j = a_j(x)$. The functions $a_j$ are in fact coefficients in a Taylor expansion along the surface $\{ (x, \omega(x)) \}$ (see e.g.~\cite{MelroseUhlmann_bookdraft}), and thus one can alternatively think that $a_j$ are functions of $(x,t)$ that satisfy $a_j(x,t) = a_j(x, \omega(x))$. Both points of view will be used in this article. In particular allowing $t$-dependence is convenient since the partial derivative ``$\partial_t$'' appears in equations \eqref{eq:Total-ACR2019} and \eqref{eq:TranEq-ACR2019}.}

According to \eqref{eq:rel14-ACR2019}, one can compute
\begin{align}
	& (\square_g + q) \big( U(x,t) \big) \nonumber\\
	= & \sum_{j \geq -1}^N (\square_g + q) a_j \cdot H_j + \sum_{j \geq -2}^{N-1} 2 C_{j+1} [\partial_t a_{j+1} + \agl[\df a_{j+1}, \df \omega] + \frac {\Delta_g \omega} 2 a_{j+1}] \cdot H_j \nonumber\\
	& + \sum_{j \geq -3}^{N-2} C_{j+2} C_{j+1} a_{j+2} [1 - \agl[\df \omega, \df \omega]] H_j + (\square_g + q) R_N. \label{eq:Total0-ACR2019}
\end{align}
Combining \eqref{eq:Total0-ACR2019} and the fact $(\square_g + q) \big( U(x,t) \big) = 0$, and analyzing the smoothness of different terms, we see that $\omega$ should satisfy the eikonal equation 
\begin{equation} \label{eq:domega1-ACR2019}
	\agl[\df \omega, \df \omega] = 1 \text{ in $B$}.
\end{equation}
Recall that one should have $\omega(x) = x_n$ when $x_n < -1$. Thus we require that 
\begin{equation} \label{eq:domega2-ACR2019}
	\omega(x) = x_n \text{ to infinite order on $\p B \cap \{ x_n < 0 \}$.}
\end{equation}

Let us assume that there exists $\omega \in C^{\infty}(\ol{B})$ satisfying \eqref{eq:domega1-ACR2019}-\eqref{eq:domega2-ACR2019} (this is true under a suitable no caustics assumption, cf.~condition \eqref{assumption1-ACR2019}). 
Then \eqref{eq:Total0-ACR2019} becomes 
\begin{align}
	& (\square_g + q) \big( U(x,t) \big) \nonumber\\
	= & \sum_{j \geq -1}^{N-1} \big\{ 2 C_{j+1} [\partial_t a_{j+1} + \agl[\df a_{j+1}, \df \omega] + \frac {\Delta_g \omega} 2 a_{j+1}] + (\square_g + q) a_j \big\} \cdot H_j \nonumber\\
	& + 2 C_{-1} [\partial_t a_{-1} + \agl[\df a_{-1}, \df \omega] + \frac {\Delta_g \omega} 2 a_{-1}] \cdot H_{-2} + (\square_g + q) a_N \cdot H_N + (\square_g + q) R_N. \label{eq:Total-ACR2019}
\end{align}
Denote
\begin{equation} \label{eq:Gammag-ACR2019}
\Gamma_g := \{ (x,t) \in \ol{B} \times \mR \,;\, t = \omega(x) \}.
\end{equation}
By analyzing the smoothness of different terms on the right-hand-side of \eqref{eq:Total-ACR2019}, 
the condition $(\square_g + q) U(x,t) = 0$ gives the following transport equations:
\begin{equation} \label{eq:TranEq-ACR2019}
	\begin{cases}
	\partial_t a_{-1} + \agl[\df \omega, \df a_{-1}] + \frac {\Delta_g \omega} 2 a_{-1} = 0, & \text{on} \quad \Gamma_g, \\
	\partial_t a_j + \agl[\df \omega, \df a_j] + \frac {\Delta_g \omega} 2 a_j = - \frac {(\square_g + q) a_{j-1}} {2j}, & \text{on} \quad \Gamma_g, \ 0 \leq j \leq N.
	\end{cases}
\end{equation}
Note that since each $a_j$ only lives on $\Gamma_g$, it is enough to solve the equations on $\Gamma_g$.

These transport equations in \eqref{eq:TranEq-ACR2019} are on the hypersurface $\Gamma_g$ and can be solved
in a way similar to solving the traditional transport equation of the form $(\p_t  + \vec a \cdot \nabla)\varphi = 0$, 
i.e.~by exploiting the integral curves of the differential operator $X := \p_t + \agl[\df \omega, \df \cdot]$ which is tangential to $\Gamma_g$. 

For later purposes we also define differential operators $X^m~(1 \leq m \leq n)$ as follows:
\begin{equation} \label{eq:vfX-ACR2019}
X^m |_{(x,t)} := \partial_j \omega(x) h^{jm}(x) \partial_t |_{(x,t)} + h^{mk}(x) \partial_k |_{(x,t)},
\quad (x,t) \in \ol{B} \times \R,
\end{equation}
where 
$h$ equals to $g^{-1/2}$, i.e.~$h$ is the unique matrix satisfying $\sum_m h^{jm} h^{mk} = g^{jk}$. (These operators are not necessary for the proofs, but they are available since we have global coordinates on $\Rn \times \R$ and some computations are easier to carry out using them.) 
The matrix $h$ is real-valued and symmetric.
From the invertibility of $h$ we see that $X^m$ is nonzero everywhere.
The notation $X^m |_{(x,t)}$ means the value of the vector field $X^m$ at point $(x,t)$.
For simplicity we also write $X^m |_{(x,t)}$ as $X_{(x,t)}^m$.
It can be checked that $X^m$ is tangential to $\Gamma_g$.

\begin{prop} \label{prop:XTang-ACR2019}
	The vector field $X^m$ defined as in \eqref{eq:vfX-ACR2019} is tangential to $\Gamma_{g}$ defined in \eqref{eq:Gammag-ACR2019}.
\end{prop}
\begin{proof}
	Let $\mathbb H(x,t) = t - \omega(x)$, then $\mathbb H$ is the defining function of the hypersurface $\Gamma_{g}$. 
	The conclusion can be reached by the following computation,
	\begin{align*}
	X^m |_{(x,\omega(x))} \mathbb H 
	& = \partial_j \omega(x) h^{jm}(x) \partial_t |_{(x,t)} \mathbb H + h^{mk}(x) \partial_k |_{(x,t)} \mathbb H \\
	& = \partial_j \omega(x) h^{jm}(x) - h^{mk}(x) \partial_k \omega(x) = 0.
	\end{align*}
	Hence $X^m$ is tangential to $\Gamma_{g}$.
\end{proof}

The vector field $X = \p_t + \agl[\df \omega, \df \cdot]$ is a linear combination of $\{X^m\}_{m = 1}^n$,
\begin{align} 
X |_{(x,t)} 
& := \partial_k \omega(x) h^{km}(x) X^m |_{(x,t)} = \partial_t |_{(x,t)} + g^{jk}(x) \partial_j \omega(x) \partial_k |_{(x,t)}.\label{eq:XDef-ACR2019} 
\end{align}
Thus $X$ is also tangential to $\Gamma_g$.
\sq{With slight ambiguity} we may also denote the restriction of $X$ to $\Gamma_g$ as $X$. 
Hence, there exist integral curves of $X$ in $\Gamma_g$ and we denote them as
\begin{equation} \label{eq:IntCur-ACR2019}
\gamma_{(x,\omega(x))} \colon t \in \R \ \mapsto \ \gamma_{(x,\omega(x))} (t) \in \Gamma_g \subset \R^{n+1},
\end{equation}
where $\gamma_{(x,\omega(x))}(0) = (x,\omega(x))$ and
$(\gamma_{(x,\omega(x))})_*(\frac {\df} {\df s} |_s) = X |_{\gamma_{(x,\omega(x))}(s)}$. 
The subscript ``${}_*$'' and superscript ``${}^*$'' stand for the push-forward and pull-back operation, respectively.
In what follows, we may omit the subscript $(x,\omega(x))$, so $\gamma(0) = (x,\omega(x))$ and $\gamma_*(\frac {\df} {\df s} |_s) = X |_{\gamma(s)}$ for short.
Now for any smooth enough function $f$ on $\Gamma_g$ we have
\begin{equation} \label{eq:XftoGamma-ACR2019}
	X |_{\gamma(s)} f = \gamma_*(\frac {\df} {\df s} \Big|_s)f = \frac {\df} {\df s} \Big|_s (f \circ \gamma).
\end{equation}

With the help of \eqref{eq:XDef-ACR2019} and \eqref{eq:XftoGamma-ACR2019}, we can rewrite \eqref{eq:TranEq-ACR2019} as
\begin{subequations} \label{eq:TranEqX-ACR2019}
	\begin{numcases}{}
	\frac {\df} {\df s} \Big|_s (a_{-1} \circ \gamma) + \frac {(\Delta_g \omega) \circ \gamma} 2 (a_{-1} \circ \gamma) = 0, \label{eq:aj-1-ACR2019} \\
	\frac {\df} {\df s} \Big|_s (a_j \circ \gamma) + \frac {(\Delta_g \omega) \circ \gamma} 2 (a_j \circ \gamma) = - \frac {[(\square_g + q) a_{j-1}] \circ \gamma} {2}, \label{eq:a0-ACR2019}
	\end{numcases}
\end{subequations}
for $0 \leq j \leq N$, where $\gamma$ denotes $\gamma_{(x,\omega(x))}$ and $\gamma(0) = (x,\omega(x))$.
Now \eqref{eq:aj-1-ACR2019} can be solved and solution is
\begin{align}
a_{-1} \circ \gamma(s) 
& = a_{-1} \circ \gamma(-\infty) \, e^{-\frac 1 2 \int_{-\infty}^s \Delta_g \omega(\gamma(\tau)) \dif \tau} \nonumber\\
& = e^{-\frac 1 2 \int_{-\infty}^s \Delta_g \omega(\gamma(\tau)) \dif \tau}. \label{eq:a-1-ACR2019}
\end{align}
\sq{Readers should note that $a_{-1}$ depends on $\omega$, but is independent of $q$.}
We used the fact $a_{-1} \circ \gamma(-\infty) = 1$.
Note that $\omega$ is independent of time $t$, and we used $\omega(\gamma(\tau))$, i.e.~$\omega(x,\omega(x))$, with slight ambiguity to represent $\omega(x)$ where $\gamma(\tau) = (x,\omega(x)) \in \mathbb{R}^n \times \R$. We use a similar convention for $q(\gamma(\tau))$.
We can solve every $a_j$ from \eqref{eq:a0-ACR2019} in a similar way as we solve $a_{-1}$ in \eqref{eq:a-1-ACR2019}.
We obtained the existence of the solution.

\subsection{Representation of the solution} \label{subsec:DevSolRie-ACR2019}

The form \eqref{eq:Uasm-ACR2019} is not the most convenient way to use to study the inverse problem.
\sq{And also, in Section \ref{subsec:ExtSolRie-ACR2019} we have seen that the coefficient $a_{-1}$ can be regarded as being independent of $t$, thus for the representation of the solution, we will rather write $a_{-1}$ bluntly as $a_{-1}(x)$.}
Now we are ready to give the formulation of the solution of the direct problem.

\begin{lem} \label{lem:dp-ACR2019}
There exists a unique solution of \eqref{eq:2-ACR2019}. In the set $B \times \R$ the solution can be represented as
\begin{equation} \label{eq:U-ACR2019}
U(x,t) = a_{-1}(x) \cdot \delta (t - \omega(x)) + u(x,t) \cdot H (t - \omega(x))
\end{equation}
where $a_{-1}$ is given by \eqref{eq:a-1-ACR2019}.
Moreover, \sq{on $\Gamma_g$,} $u$ satisfies
\begin{align} 
u(x,\omega(x))
& = -\frac 1 2 a_{-1}(x) \int_{-\infty}^0 \big[ \frac {-\Delta_g a_{-1}(\gamma(\tau))} {a_{-1}(\gamma(\tau))} + q(\gamma(\tau)) \big] \dif \tau, \label{eq:uExp-ACR2019} \\
X |_{(x,\omega(x))} \left( \frac {u} {a_{-1}} \right) (x,\omega(x))
& = - \frac 1 2 \left( \frac {-\Delta_g a_{-1}(x)} {a_{-1}(x)} + q(x) \right), \label{eq:Xgasu-ACR2019}
\end{align}
where the $\gamma$ is defined in \eqref{eq:IntCur-ACR2019}.
\end{lem}

\begin{proof}
The distribution $U(x,t) = a_{-1}(x) \delta(t-x_n) + V(x,t)$ solves \eqref{eq:2-ACR2019} if and only if $V$ solves 
\[
(\p_t^2 - \Delta_g + q) V = -(\p_t^2 - \Delta_g + q) \big( a_{-1}(x) \delta(t-x_n) \big)
\]
with $V = 0$ for $t < -1$. 
The right hand side is supported in $\{ t \geq -1 \}$, and hence the existence of a unique solution $V$ is guaranteed by \cite[Theorem 23.2.4]{Hormander3}.
With the help of \eqref{eq:aj-1-ACR2019}, from \eqref{eq:U-ACR2019} we obtain
\begin{align}
	(\square_g + q) \big( U(x,t) \big)
	& = (\square_g + q) u \cdot H  + [ 2\partial_t u + 2\agl[\df \omega, \df u] + \Delta_g \omega u + (\square_g + q) a_{-1} ]\cdot \delta. \label{eq:Total2-ACR2019}
\end{align}
Combining \eqref{eq:Total2-ACR2019} with \eqref{eq:XDef-ACR2019} and
\eqref{eq:XftoGamma-ACR2019}, and analyzing the smoothness of different terms on the right-hand-side of \eqref{eq:Total2-ACR2019}, 
the condition $(\square_g + q) U(x,t) = 0$ gives
\begin{subequations} \label{eq:TranEq2-ACR2019}
	\begin{numcases}{}
	\frac {\df} {\df s} \Big|_s (u \circ \gamma) + \frac {(\Delta_g \omega) \circ \gamma} 2 (u \circ \gamma) = - \frac {[(-\Delta_g + q) a_{-1}] \circ \gamma} {2}, \label{eq:b1-ACR2019} \medskip\\
	(\square_g + q) u = 0, \quad \text{in} \quad t > \omega(x). \label{eq:b2-ACR2019}
	\end{numcases}
\end{subequations}

Denote $u_\gamma := u \circ \gamma$, thus \eqref{eq:b1-ACR2019} gives
\begin{equation} \label{eq:vODE-ACR2019}
	 \frac {\df u_\gamma} {\df s} (s)  + \frac 1 2 \Delta_g \omega(\gamma(s)) \, u_\gamma(s) = -\frac 1 2 (\square_g + q) a_{-1}(\gamma(s)).
\end{equation}
Solving \eqref{eq:vODE-ACR2019} with the help of \eqref{eq:a-1-ACR2019} and the initial condition $u(x,t) = 0$ when $t \ll -1$, we arrive at
\begin{equation*}
u_\gamma(s) 
= -\frac 1 2 (a_{-1} \circ \gamma)(s) \int_{-\infty}^s \big[ \frac {-\Delta_g a_{-1}(\gamma(\tau))} {a_{-1}(\gamma(\tau))} + q(\gamma(\tau)) \big] \dif \tau,
\end{equation*}
and hence
\begin{equation*}
X |_{\gamma(s)} \left( \frac {u} {a_{-1}} \right)
= - \frac 1 2 \left( \frac {-\Delta_g a_{-1}(\gamma(s))} {a_{-1}(\gamma(s))} + q(\gamma(s)) \right).
\end{equation*}
The proof is complete.
\end{proof}

\section{Some energy estimates} \label{sec:energy-ACR2019}

In this section we prove several energy estimates, which extend corresponding results in \cite{RSfixed2019March, RSfixed2019} to the case of Riemannian metrics. Before that, we define some domains which we work on, 
and then present some preliminary lemmas.

\subsection{Geometry and preliminary lemmas} \label{subsec:geopre-ACR2019}

The underlying geometry and related domains are listed below.

\medskip

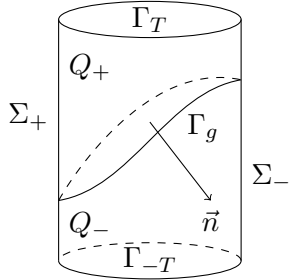
\begin{tabular}{lr}

	\begin{tikzpicture}[scale = 0.8, baseline=(current bounding box.center)]
	\pgfmathsetmacro{\LONG}{1.5}; 
	\pgfmathsetmacro{\SHORT}{0.3}; 
	\pgfmathsetmacro{\HEIGHT}{4}; 
	\pgfmathsetmacro{\LP}{1}; 
	\pgfmathsetmacro{\RP}{3}; 
	\coordinate (BL) at (-\LONG,0);
	\coordinate (BR) at (\LONG,0);
	\coordinate (TL) at (-\LONG,\HEIGHT);
	\coordinate (TR) at (\LONG,\HEIGHT);
	\draw (0,\HEIGHT) ellipse ({\LONG} and {\SHORT});
	\draw[dashed] (\LONG,0) arc(0:180:{\LONG} and {\SHORT});
	\draw (-\LONG,0) arc(180:360:{\LONG} and {\SHORT});
	\draw (BL) -- (TL);
	\draw (BR) -- (TR);
	\draw (-\LONG,\LP) .. controls (-0.2*\LONG,0.9*\LP + 0.1*\RP) and (0.2*\LONG,0.1*\LP + 0.9*\RP) .. ({\LONG},\RP);
	\draw[dashed] (-\LONG,\LP) .. controls (-0.7*\LONG,0.6*\LP + 0.4*\RP) and (0.1*\LONG,1.1*\RP) .. ({\LONG},\RP);
	\node[anchor=north west] at (-\LONG,0.9*\HEIGHT) {$Q_+$};
	\node[anchor=west] at (0.3*\LONG,1.1*\HEIGHT/2) {$\Gamma_{g}$};
	\node[anchor=south east] at (-\LONG,\LP*2) {$\Sigma_{+}$};
	\node at (0,\HEIGHT) {$\Gamma_{T}$};
	\node[anchor=south west] at (-\LONG,0.05*\HEIGHT) {$Q_-$};
	\node[anchor=south west] at (\LONG,\LP) {$\Sigma_{-}$};
	\node at (0,0) {$\Gamma_{-T}$};
	\draw[->] (0,1.15*\HEIGHT/2) -- (1,\HEIGHT/2-1) node[anchor=north]{$\vec n$};
	\end{tikzpicture}
	\label{fig:geo-ACR2019}

& \qquad \qquad

{$\begin{aligned}
& Q := B \times [-T,T], 
\quad \Sigma := \p B \times [T,T], \\
& Q_\pm := Q \cap \{ (x,t) \,;\, \pm(t - \omega(x)) > 0 \}, \\
& \Sigma_\pm := \Sigma \cap \{(x,t) \,;\, \pm(t - \omega(x)) > 0 \}, \\
& \Gamma_g = Q \cap \{ (x,t) \,;\, t = \omega(x) \}, \\
& \Gamma_{\pm T} := Q \cap \{ t = \pm T \}, \quad \Gamma_{\tau} := Q \cap \{ t = \tau \}, \\
& \vec n := (\nabla \omega(x),-1)/\sqrt{|\nabla \omega(x)|^2 + 1}.
\end{aligned}$}

\end{tabular}


The precise value of $T > 0$ shall be determined later. We will also consider a Riemannian metric $\tilde{g} = g \oplus 1$ on $\R^n \times \R$, and we denote by $g'$ the metric on $\Gamma_g$ induced by $\tilde{g}$.
Before we show the energy estimates, we first present some lemmas which are used in the proof of these energy estimates.

\begin{lem} \label{lem:Claim1Equ-ACR2019}
	For smooth enough functions $\alpha$ and $\beta$, we have three different representations of $\sum_{m = 1}^n (X^m \alpha) (X^m \beta)$,
	\begin{align}
		\sum_{m = 1}^n (X^m \alpha) (X^m \beta)
		& = \agl[\df \alpha,\df \beta] + \alpha_t \beta_t + \omega_j g^{jk} (\alpha_k \beta_t + \beta_k \alpha_t), \label{eq:Claim1Equ-ACR2019} \\
		\sum_{m = 1}^n (X^m \alpha) (X^m \beta)
		& = \agl[\df \alpha,\df \beta] - \alpha_t \beta_t + \alpha_t (X\beta) + \beta_t (X\alpha), \label{eq:Claim1Equ2-ACR2019} \\
\sum_{m = 1}^n (X^m \alpha) (X^m \beta)
& = (X \alpha)(X \beta) + \langle\df \alpha - \agl[\df \omega, \df \alpha] \df \omega, \df \beta - \agl[\df \omega, \df \beta \rangle \df \omega]. \label{eq:Claim1Equ3-ACR2019}
\	\end{align}
\end{lem}
\begin{proof}
	We compute the left-hand-side of \eqref{eq:Claim1Equ-ACR2019} as follows,
	\begin{align*}
	\sum_m (X^m \alpha) (X^m \beta)
	& = \sum_m (\omega_j h^{jm} \alpha_t + h^{mk} \alpha_k) (\omega_j h^{jm} \beta_t + h^{mk} \beta_k) \\
	& = \omega_j g^{jk} \omega_k \alpha_t \beta_t + \omega_j g^{jk} \alpha_t \beta_k + \omega_j g^{jk} \beta_t \alpha_k + \alpha_j g^{jk} \beta_k \\
	& = \alpha_t \beta_t + \agl[\df \alpha,\df \beta] + \omega_j g^{jk} (\alpha_k \beta_t + \beta_k \alpha_t).
	\end{align*}
	Note that we used the fact \eqref{eq:domega1-ACR2019}.
	The equation \eqref{eq:Claim1Equ-ACR2019} is proved.
	Then from \eqref{eq:Claim1Equ-ACR2019}, we can obtain \eqref{eq:Claim1Equ2-ACR2019} as follows,
	\begin{align*}
	\agl[\df \alpha,\df \beta] - \alpha_t \beta_t
	& = \sum_m (X^m \alpha) (X^m \beta) - 2\alpha_t \beta_t - \omega_j g^{jk} (\alpha_k \beta_t + \beta_k \alpha_t) \\
	& = \sum_m (X^m \alpha) (X^m \beta) - \alpha_t (\beta_t + \omega_j g^{jk} \beta_k) - \beta_t (\alpha_t + \omega_j g^{jk} \alpha_k) \\
	& = \sum_m (X^m \alpha) (X^m \beta) - \alpha_t (X\beta) - \beta_t (X\alpha).
	\end{align*}
\sq{To show} \eqref{eq:Claim1Equ3-ACR2019}, by \eqref{eq:Claim1Equ-ACR2019} and noting that $\agl[\df \omega, \df \omega] = 1$, we have	
\begin{align*}
\sum_{m = 1}^n (X^m \alpha) (X^m \beta)
& = \agl[\df \alpha,\df \beta] + (X \alpha)(X \beta) - \agl[\df \omega, \df \alpha] \agl[\df \omega, \df \beta], \\
& = (X \alpha)(X \beta) + \langle\df \alpha - \agl[\df \omega, \df \alpha] \df \omega, \df \beta - \agl[\df \omega, \df \beta \rangle \df \omega].
\end{align*}
	The proof is complete.
\end{proof}

\begin{lem} \label{lem:Claim2Equ-ACR2019}
	It holds \sq{that}
	\begin{equation} \label{eq:Claim2Equ-ACR2019}
	\sum_m |X^m \alpha|^2 \leq C \agl[\df \alpha,\df \overline \alpha]_{g'},
	\end{equation}
	where the $g'$ is the induced Riemannian metric over the hypersurface $\Gamma_g$, and the constant $C$ is independent of $\alpha$.
\end{lem}


\begin{proof}
	For a fixed $m$ and a fixed point $p \in \Gamma_g$, we know $X_p^m \in T_p \Gamma_g$, thus we can find local coordinates $\{y^1,y^2,\cdots,y^{n-1}\}$ such that 
	\begin{equation*}
	\begin{cases}
		\partial_{y_1} |_p = X_p^m \\
		g'(\partial_{y^i} |_p,\partial_{y^j} |_p) = g'_{ij}(p) = \delta_{ij}
		\quad \Rightarrow \quad
		g'(\df y^j |_p,\df y^k |_p) = g'^{jk}(p) = \delta^{jk}.
	\end{cases}
	\end{equation*}
	\sq{We see that the set $\{ X^m |_p, \partial_{y^j} |_p \,;\, j = 2,\cdots,n-1 \}$ forms an orthogonal basis of $T_p \Gamma_g$ according to the metric $g'$.}
	The exterior-derivative of $\alpha$ at point $p$ on $\Gamma_g$ is $\df \alpha |_p = X_p^m \alpha + \sum_{j=2}^{n-1} (\partial_{y^j} \alpha) |_p \dif y^j$, so	
	\begin{equation*}
		|X_p^m \alpha|^2 
		\leq |X_p^m \alpha|^2 + \sum_{j=2}^{n-1} |\partial_{y^j} \alpha|^2
		= |X_p^m \alpha|^2 + \sum_{j,k=2}^{n-1} \partial_{y^j} \alpha \cdot g'^{jk}(p) \cdot \overline{\partial_{y^k} \alpha}
		\lesssim \agl[\df \alpha |_p, \overline{\df \alpha} |_p]_{g'},
	\end{equation*}
	\sq{which yields} \eqref{eq:Claim2Equ-ACR2019}.
\end{proof}

We use the following convention for integrations and Sobolev norms on $(\Rn \times \R_t,\tilde{g})$ and on $(\Gamma_g,g')$.
Below, we assume that $\mathcal P \subset \Rn \times \R$ and $\mathcal Q \subset \Gamma_g$.
\begin{align*}
& \int_{\mathcal P} f \dif V_g := \int_{\mathcal P} f(x,t) |\tilde g(x,t)|^{1/2}\dif x \dif t,
&&
\int_{\mathcal Q} f \dif S_{g'} := \int_{\mathcal Q} f(x) |g'(x)|^{1/2}\dif S, \\
& \nrm[L^2(\mathcal P; g)]{f} := (\int_{\mathcal P} |f|^2 \dif V_g)^{1/2}, 
&& \nrm[L^2(\mathcal Q; g')]{f} := (\int_{\mathcal Q} |f|^2 \dif S_{g'})^{1/2}, \\
& \nrm[H^1(\mathcal P; g)]{f} := (\int_{\mathcal P} |f|^2 + \agl[\df f, \overline{\df f}] \dif V_g)^{1/2},
&&
\nrm[H^1(\mathcal Q; g')]{f} := (\int_{\mathcal Q} |f|^2 + \agl[\df f, \overline{\df f}]_{g'} \dif S_{g'})^{1/2}.
\end{align*}
Here the volume form $\df V_g$ corresponds to the space-time tensor $\tilde g:= g \oplus 1$, and $\df S$ is the volume form on $\Gamma_g$ induced by the usual Euclidean metric.

\begin{lem} \label{lem:IntEqu-ACR2019}
	\sq{There exist two positive constants $C_1$, $C_2$ such that,} for any non-negative function $f$ and any bounded hypersurface $\mathcal Q \subset \Gamma_g$, we have
	\begin{equation} \label{eq:IntEqu-ACR2019}
	C_1 \int_{\mathcal Q} f(p) \dif S_{g'}
	\leq \int_{\mathcal Q} f(p) \dif S
	\leq C_2 \int_{\mathcal Q} f(p) \dif S_{g'}.
	\end{equation}
\end{lem}

\begin{proof}
	Because $\mathcal Q$ is bounded, we can find finite many local charts $(U_i,\varphi_i)$ and partition of unity $\{\chi_i\}$ such that
	\begin{equation*}
	\begin{cases}
	\supp \chi_i \subset U_i,
	\quad \sum_i \chi_i = 1, \quad \overline{\mathcal Q} \subset \cup_i U_i, \\
	\{ \partial_{(\varphi_i)^j} |_p \}_{j=1}^{n-1} \text{~forms a positve basis of~} T_p \mathcal Q, \text{~under the metric~} g', \\
	\{ \partial_{(\varphi_i)^j} |_p \}_{j=1}^{n-1} \text{~forms a orthonormal positve basis of~} T_p \mathcal Q, \text{~under the Euclidean metric}.
	\end{cases}
	\end{equation*}
	In these $U_i$, the volume form is $\df S_g = |g'|^{1/2} \dif (\varphi_i)^1 \wedge \cdots \wedge \dif (\varphi_i)^{n-1}$ where the $jk$-th component of the matrix $g'$ is $(g')^{jk} |_p = \agl[\dif (\varphi_i)^j, \dif (\varphi_i)^k]$.
	Thus, $|g| / \lambda_1 \leq |g'| \leq |g| / \lambda_2$ where the $\lambda_1$ and $\lambda_2$ are the largest and smallest eigenvalues of $g$, respectively.
	If we denote the supremum and infimum of this $|g'|$ on $U_i$ as $C_{i,1}$ and $C_{i,2}$, then $C_{i,1}$, $C_{i,2} > 0$ and
	\begin{align*}
	\int_{\mathcal Q} f(p) \dif S_{g'} 
	& = \sum_i \int_{\mathcal Q} f(p) \chi_i(p) |g'|^{1/2} \dif (\varphi_i)^1 \wedge \cdots \dif \wedge (\varphi_i)^{n-1} \\
	& \leq \sum_i C_{i,1} \int_{\mathcal Q} f(p) \chi_i(p) \dif (\varphi_i)^1 \wedge \cdots \dif \wedge (\varphi_i)^{n-1} \\
	& \lesssim \int_{\mathcal Q} f(p) \sum_i \chi_i(p) \dif (\varphi_i)^1 \wedge \cdots \dif \wedge (\varphi_i)^{n-1} \\
	& = \int_{\mathcal Q} f(p) \sum_i \chi_i(p) \dif S
	= \int_{\mathcal Q} f(p) \dif S,
	\end{align*}
	and
	\begin{align*}
	\int_{\mathcal Q} f(p) \dif S_{g'} 
	& = \sum_i \int_{\mathcal Q} f(p) \chi_i(p) |g'|^{1/2} \dif (\varphi_i)^1 \wedge \cdots \dif \wedge (\varphi_i)^{n-1} \\
	& \geq \sum_i C_{i,2} \int_{\mathcal Q} f(p) \chi_i(p) \dif (\varphi_i)^1 \wedge \cdots \dif \wedge (\varphi_i)^{n-1} \\
	& \gtrsim \int_{\mathcal Q} f(p) \sum_i \chi_i(p) \dif (\varphi_i)^1 \wedge \cdots \dif \wedge (\varphi_i)^{n-1} \\
	& = \int_{\mathcal Q} f(p) \sum_i \chi_i(p) \dif S
	= \int_{\mathcal Q} f(p) \dif S.
	\end{align*}
	The proof is complete.
\end{proof}

\subsection{Energy estimates} \label{subsec:enEst-ACR2019}

We consider a parameter $\sigma \geq 1$. All constants in this article will be independent of $\sigma$. Throughout the paper, we use the notation $\calX = \calX_{\sigma}$ defined as follows,
\begin{equation} \label{eq:calX-ACR2019}
	\calX f := \frac {\sigma^2 |f|^2 + \sum_{m= 1}^n |X^m f|^2} {\sqrt{|\nabla \omega(x)|^2+1}},
\end{equation}
for any function for which the right-hand-side of \eqref{eq:calX-ACR2019} can be defined.
We first give two energy estimates, which are straightforward extensions of \cite[Lemmas 3.4 and 3.5]{RSfixed2019March} to the Riemannian case.

\begin{lem} \label{lem:energyTop-ACR2019}
	For $\sigma = 1$ we have
	\begin{align}
	\int_{\Gamma_\tau} \big( |\alpha|^2 + |\alpha_t|^2 + \agl[ \df \alpha, \overline{\df \alpha} ] \big) |g|^{1/2} \dif S
	& \leq C( \nrm[L^2(Q_+;g)]{(\square_g + q)\alpha}^2 + \int_{\Gamma_{g}} (\calX_1 \alpha) |g|^{1/2} \dif S \nonumber\\
	& \quad + \nrm[H^1(\Sigma_{+};g')]{\alpha}^2 + \nrm[L^2(\Sigma_{+};g')]{\partial_{\nu,g} \alpha}^2 ) \label{eq:energyTop-ACR2019}
	\end{align}
	where the constant $C$ depends on $\nrm[L^\infty]{q}$, and $\partial_{\nu,g} \alpha := \nu_j g^{jk}(x) \partial_k \alpha$.
\end{lem}

\begin{proof}
	Denote $Q_\tau := \{ (x,t) \,;\, |x| \leq 1, \omega(x) \leq t \leq \tau\}$.
	The outer normal of $\Gamma_g$ is $\frac {(\nabla \omega(x),-1)} {\sqrt{|\nabla \omega(x)|^2+1}}$ pointing into $\{ (x,t) \in \Rn \times \R \,;\, t < \omega(x) \}$. Let $\Gamma_{g,\tau} := \Gamma_g \cap \{ t \leq \tau \}$ and $\Sigma_{+,\tau} := \Sigma_+ \cap \{ t \leq \tau \}$.
	
	Let $\sigma \geq 1$ (later we will take $\sigma=1$). We have
	\begin{equation} \label{eq:t.identity-ACR2019}
	\overline{\alpha_t} (\square_g \alpha + \sigma^2 \alpha) + \alpha_t (\overline{\square_g \alpha + \sigma^2 \alpha}) 
	= \partial_t \big( \sigma^2 |\alpha|^2 + |\alpha_t|^2 + \agl[ \df \alpha, \overline{\df \alpha} ] \big) - \tilde \delta (\overline{\alpha_t} \dif \alpha) - \tilde \delta (\alpha_t \overline{\dif \alpha}).
	\end{equation}

	The term
	\[
	f_\sigma(\tau) := \int_{\Gamma_\tau} \big(  \sigma^2  |\alpha|^2 + |\alpha_t|^2 + \agl[ \df \alpha, \overline{\df \alpha} ] \big) |g|^{1/2} \dif S
	\]
	is the quantity that we wish to estimate.
	Integrating \eqref{eq:t.identity-ACR2019} over $Q_\tau$, we have 
	\begin{align}
	& \int_{Q_\tau} \overline{\alpha_t} (\square_g \alpha + \sigma^2 \alpha) + \alpha_t (\overline{\square_g \alpha + \sigma^2 \alpha}) \dif V_g \nonumber\\
	= & \int_{Q_\tau} \partial_t \big( \sigma^2 |\alpha|^2 + |\alpha_t|^2 + \agl[ \df \alpha, \overline{\df \alpha} ] \big) - \tilde \delta (\overline{\alpha_t} \dif \alpha) - \tilde \delta (\alpha_t \overline{\dif \alpha}) \dif V_g \nonumber\\
	= & \int_{\Gamma_\tau} ( \sigma^2 |\alpha|^2 + |\alpha_t|^2 + \agl[ \df \alpha, \overline{\df \alpha} ] ) |g|^{1/2} \dif S
	- \int_{\Gamma_{g,\tau}} \frac {\sigma^2 |\alpha|^2 + |\alpha_t|^2 + \agl[ \df \alpha, \overline{\df \alpha} ]} {\sqrt{|\nabla \omega(x)|^2+1}} |g|^{1/2} \dif S \nonumber\\
	& \ - \int_{\Gamma_{g,\tau} \,\cup\, \Sigma_{+,\tau}} \nu_j g^{jk} (\overline{\alpha_t} \alpha_k + \alpha_t \overline{\alpha_k}) |g|^{1/2} \dif S \nonumber\\
	= & \ f_\sigma(\tau) - \int_{\Gamma_{g,\tau}} \frac {\sigma^2 |\alpha|^2 + |\alpha_t|^2 + \agl[\df \alpha, \overline{\df \alpha}] + (\partial_j \omega) g^{jk} (\overline{\alpha_t} \alpha_k + \alpha_t \overline{\alpha_k})} {\sqrt{|\nabla \omega(x)|^2+1}} |g|^{1/2} \dif S \nonumber\\
	& \ - \int_{\Sigma_{+,\tau}} [(\partial_{\nu,g} \alpha) \overline{\alpha_t} + (\overline{\partial_{\nu,g} \alpha}) \alpha_t] |g|^{1/2} \dif S. \label{eq:ident1-ACR2019}
	\end{align}
	We now let $\sigma=1$.
	Recall the definition of $X^m$ in \eqref{eq:vfX-ACR2019}.
	Combining \eqref{eq:ident1-ACR2019} with
	\eqref{eq:Claim1Equ-ACR2019} and \eqref{eq:calX-ACR2019}, we can estimate $f_1(\tau)$ as
	\begin{align*}
	f_1(\tau)
	& = \int_{Q_\tau} \overline{\alpha_t} (\square_g \alpha +     \alpha) + \alpha_t (\overline{\square_g \alpha +     \alpha}) \dif V_g + \int_{\Gamma_{g,\tau}} (\calX_1 \alpha) |g|^{1/2} \dif S \nonumber\\
	& \quad + \int_{\Sigma_{+,\tau}} [(\partial_{\nu,g} \alpha) \overline{\alpha_t} + (\overline{\partial_{\nu,g} \alpha}) \alpha_t] |g|^{1/2} \dif S \nonumber\\
	& \leq \int_{Q_\tau} (|\alpha_t|^2 + |\square_g \alpha +     \alpha|^2) \dif V_g + \int_{\Gamma_{g,\tau}} (\calX_1 \alpha) |g|^{1/2} \dif S + \int_{\Sigma_{+,\tau}} (|\alpha_t|^2 + |\partial_{\nu,g} \alpha|^2) |g|^{1/2} \dif S \nonumber\\
	& \leq \int_{Q_\tau} [|\alpha_t|^2 + 2 (1 + \nrm[L^\infty]{q})^2 |\alpha|^2 + 2|(\square_g + q) \alpha|^2] \dif V_g + \int_{\Gamma_{g,\tau}} (\calX_1 \alpha) |g|^{1/2} \dif S \nonumber\\
	& \quad + \int_{\Sigma_{+,\tau}} (|\alpha_t|^2 + |\partial_{\nu,g} \alpha|^2) |g|^{1/2} \dif S \nonumber\\
	& \leq C\int_{-1}^\tau f_1(t) \dif t + \int_{Q_\tau} |(\square_g + q) \alpha|^2 \dif V_g + \int_{\Gamma_{g,\tau}} (\calX_1 \alpha) |g|^{1/2} \dif S \nonumber\\
	& \quad + \int_{\Sigma_{+,\tau}} (|\alpha_t|^2 + |\partial_{\nu,g} \alpha|^2) |g|^{1/2} \dif S,
	\end{align*}
	and the Gr\"{o}nwall's inequality gives
	\begin{equation} \label{eq:fwoSg-ACR2019}
	f_1(\tau) 
	\lesssim \int_{Q_\tau} |(\square_g + q) \alpha|^2 \dif V_g + \int_{\Gamma_{g,\tau}} (\calX_1 \alpha) |g|^{1/2} \dif S + \int_{\Sigma_{+,\tau}} (|\alpha_t|^2 + |\partial_{\nu,g} \alpha|^2) |g|^{1/2} \dif S.
	\end{equation}

	Now combining \eqref{eq:fwoSg-ACR2019} with \eqref{eq:IntEqu-ACR2019}, we arrive at \eqref{eq:energyTop-ACR2019}.
\end{proof}

\begin{lem} \label{lem:energyBottom-ACR2019}
	For any functions $f, \varphi, \alpha \in C^2(\overline{Q_+})$, we have
	\begin{align}
		& \int_{\Gamma_g} e^{2\sigma \varphi} (\calX (f\alpha)) |g|^{1/2} \dif S \nonumber\\
		\leq & \ C_{q,f} \big[ \nrm[L^2(Q_+;g)]{e^{\sigma \varphi} (\square_g + q) \alpha}^2 + \sigma \nrm[L^2(Q_+;g)]{e^{\sigma \varphi} \nabla_{x,t} \alpha}^2 + \sigma^3 \nrm[L^2(Q_+;g)]{e^{\sigma \varphi} \alpha}^2 \nonumber\\
		& \ + \nrm[L^2(\Sigma_{+})]{e^{\sigma \varphi} \partial_{\nu} \alpha}^2 + \nrm[L^2(\Sigma_{+})]{e^{\sigma \varphi} \alpha_t}^2 + \sigma^2 \nrm[L^2(\Sigma_{+})]{e^{\sigma \varphi} \alpha}^2 \big]. \label{eq:energyBottom-ACR2019}
	\end{align}
	where the constant $C_{q,f}$ depends on $\nrm[L^\infty]{q}$ and $f$ and the metric $g$.
\end{lem}
\begin{proof}
	We first focus on the case where $f = 1$.	
	From \eqref{eq:ident1-ACR2019} and \eqref{eq:Claim1Equ-ACR2019}, we can have
	\begin{align*}
	& \int_{Q_\tau} \overline{\beta_t} (\square_g \beta + \sigma^2 \beta) + \beta_t (\overline{\square_g \beta + \sigma^2 \beta}) \dif V_g \nonumber\\
	= & \ f_\sigma(\tau) - \int_{\Gamma_{g,\tau}} (\calX \beta) |g|^{1/2} \dif S - \int_{\Sigma_{+,\tau}} [(\partial_{\nu,g} \beta) \overline{\beta_t} + (\overline{\partial_{\nu,g} \beta}) \beta_t] |g|^{1/2} \dif S,
	\end{align*}
	where $f_\sigma(\tau) = \int_{\Gamma_\tau} \big( \sigma^2 |\beta|^2 + |\beta_t|^2 + (\df \beta, \overline{\df \beta})_g \big) |g|^{1/2} \dif S$, thus
	\begin{align}
	\int_{\Gamma_{g,\tau}} (\calX \beta) |g|^{1/2} \dif S
	& = f_\sigma(\tau) - \int_{Q_\tau} \overline{\beta_t} (\square_g \beta + \sigma^2 \beta) + \beta_t (\overline{\square_g \beta + \sigma^2 \beta}) \dif V_g \nonumber\\
	& \quad - \int_{\Sigma_{+,\tau}} [(\partial_{\nu,g} \beta) \overline{\beta_t} + (\overline{\partial_{\nu,g} \beta}) \beta_t] |g|^{1/2} \dif S. \label{eq:3.2Temp1-ACR2019}
	\end{align}

	Now let $\beta = e^{\sigma \varphi} \alpha$, then
	\begin{equation} \label{eq:alphabeta-ACR2019}
	\left\{\begin{aligned}
	X^m \beta 
	& = e^{\sigma \varphi} (X^m + \sigma X^m \varphi)\alpha, \\
	\partial_t \beta
	& = e^{\sigma \varphi} (\alpha_t + \sigma \varphi_t \alpha), \quad \partial_t^2 \beta
	= e^{\sigma \varphi} \{ \alpha_{tt} + 2\sigma \varphi_t \alpha_t + [\sigma \varphi_{tt} + (\sigma \varphi_t)^2] \alpha \} \\
	\Delta_g \beta
	& = \Delta_g (e^{\sigma \varphi}) \cdot \alpha + 2\agl[\df (e^{\sigma \varphi}), \df \alpha] + e^{\sigma \varphi} \cdot \Delta_g \alpha \\
	\Delta_g (e^{\sigma \varphi})
	& = e^{\sigma \varphi} [\sigma \Delta_g \varphi + \sigma^2 \agl[\df \varphi, \df \varphi]] \\
	\agl[\df (e^{\sigma \varphi}), \df \alpha]
	& = \sigma e^{\sigma \varphi} \agl[\df \varphi, \df \alpha]
	\end{aligned}\right.
	\end{equation}
	so
	\begin{align*}
	(\square_g + q) \beta
	& = e^{\sigma \varphi} \{ (\square_g + q) \alpha + 2\sigma [\varphi_t \alpha_t - \agl[\df \varphi, \df \alpha]] + [\sigma \square_g \varphi + \sigma^2 \varphi_t^2 - \sigma^2 \agl[\df \varphi, \df \varphi]] \alpha \} \\
	& = e^{\sigma \varphi} \{ (\square_g + q) \alpha + 2\sigma \agl[\df \varphi, \df \alpha]_{\hat g} + [\sigma \square_g \varphi + \sigma^2 \agl[\df \varphi, \df \varphi]_{\hat g}] \alpha \},
	\end{align*}
	where $\hat g = (\df t)^2 - g$.

	It can be checked that
	\[
	|\agl[\df F,\df G]| \leq \sqrt{\agl[\df F,\df \overline F]} \cdot \sqrt{\agl[\df G,\df \overline G]} \leq C_g |\nabla_x F| \cdot |\nabla_x G|
	\]
	for some constant $C_g$ and for any functions $F$, $G \in C^1(\Rn;\mathbb C)$.
	When $\sigma$ is large enough, i.e.~$\sigma^2 > 2\nrm[L^\infty]{q}$, we compute the integral of $\overline{\beta_t} (\square_g \beta + \sigma^2 \beta)$ as follows,
	\begin{align}
	& \int_{Q_\tau} \overline{\beta_t} (\square_g \beta + \sigma^2 \beta) \dif V_g \nonumber\\
	\lesssim & \int_{Q_\tau} \overline{\beta_t} \cdot (\square_g + q) \beta + (\sigma^2 - q) \overline{\beta_t} \beta \dif V_g \nonumber\\
	= & \int_{Q_\tau} \overline{\beta_t} \cdot e^{\sigma \varphi} \big\{ (\square_g + q) \alpha + 2\sigma \agl[ \df \varphi, \df \alpha ]_{\hat g} + [\sigma \square_g \varphi + \sigma^2 \agl[\df \varphi, \df \varphi]_{v g}] \alpha \big\} + (\sigma^2 - q) \overline{\beta_t} \beta \dif V_g \nonumber\\
	\lesssim & \int_{Q_\tau} |\beta_t| \cdot e^{\sigma \varphi} \big[ |(\square_g + q) \alpha| + |\nabla_{x,t} \varphi| \cdot \sigma |\nabla_{x,t} \alpha| +  (|\nabla_{x,t} \varphi| + |\square_g \varphi|) \cdot \sigma^2 |\alpha| \big] \nonumber\\
	& \quad + \sigma |\beta_t| \cdot \sigma |\beta| \dif V_g \nonumber\\
	\lesssim & \ C_\varphi \int_{Q_\tau} (|\beta_t|^2 + e^{2\sigma \varphi} |(\square_g + q) \alpha|^2) + (\sigma |\beta_t|^2 + \sigma e^{2\sigma \varphi} |\nabla_{x,t} \alpha|^2) \nonumber\\
	& \quad + ( \sigma |\beta_t|^2 + \sigma^3 e^{2\sigma \varphi} |\alpha|^2 ) + (\sigma |\beta_t|^2 + \sigma^3 |\beta|^2) \dif V_g \nonumber\\
	\lesssim & \int_{Q_\tau} e^{2\sigma \varphi} \big( |(\square_g + q) \alpha|^2 + \sigma |\nabla_{x,t} \alpha|^2 + \sigma^3 \alpha^2 \big) + \sigma |\beta_t|^2 + \sigma^3 |\beta|^2 \dif V_g \nonumber\\
	\lesssim & \int_{Q_\tau} e^{2\sigma \varphi} \big( |(\square_g + q) \alpha|^2 + \sigma |\nabla_{x,t} \alpha|^2 + \sigma^3 \alpha^2 \big) + \sigma e^{2\sigma \varphi} (|\alpha_t|^2 + \sigma^2 |\alpha|^2) + \sigma^3 e^{2\sigma \varphi} |\alpha|^2 \dif V_g \nonumber\\
	\lesssim & \ \nrm[L^2(Q_\tau;g)]{e^{\sigma \varphi} (\square_g + q) \alpha}^2 + \sigma \nrm[L^2(Q_\tau;g)]{e^{\sigma \varphi} \nabla_{x,t} \alpha}^2 + \sigma^3 \nrm[L^2(Q_\tau;g)]{e^{\sigma \varphi} \alpha}^2. \label{eq:3.2Temp2-ACR2019}
	\end{align}
	Similarly, we have
	\begin{align}
	& \int_{Q_\tau} \beta_t (\overline{\square_g \beta + \sigma^2 \beta}) \dif V_g \nonumber\\
	\leq & \ C_\varphi ( \nrm[L^2(Q_\tau;g)]{e^{\sigma \varphi} (\square_g + q) \alpha}^2 + \sigma \nrm[L^2(Q_\tau;g)]{e^{\sigma \varphi} \nabla_{x,t} \alpha}^2 + \sigma^3 \nrm[L^2(Q_\tau;g)]{e^{\sigma \varphi} \alpha}^2 ). \label{eq:3.2Temp3-ACR2019}
	\end{align}
	Also, the $f_\sigma(\tau)$ can be estimated as
	\begin{align}
	f_\sigma(\tau) 
	& = \int_{\Gamma_\tau} \big( \sigma^2 |\beta|^2 + |\beta_t|^2 + \agl[\df \beta, \overline{\df \beta}] \big) |g|^{1/2} \dif S \nonumber\\
	& \leq C_\varphi \int_{\Gamma_\tau} \sigma^2 e^{2\sigma \varphi} |\alpha|^2 + e^{2\sigma \varphi} (|\nabla_{x,t} \alpha|^2 + \sigma^2 |\alpha|^2) \dif S_g \nonumber\\
	& \leq C_\varphi (\nrm[L^2(\Gamma_\tau;g)]{e^{\sigma \varphi} \nabla_{x,t} \alpha}^2 + \sigma^2 \nrm[L^2(\Gamma_\tau;g)]{e^{\sigma \varphi} \alpha}^2). \label{eq:3.2Temp4-ACR2019}
	\end{align}
	Note the metric $g$ \sq{equals} to the usual Euclidean metric on the boundary $\Sigma_+ \cup \Sigma_-$, so $\partial_{\nu,g}$ \sq{equals} to $\partial_{\nu}$.
	Then the integral on the boundary $\Sigma_{+,\tau}$ is
	\begin{align*}
	& \int_{\Sigma_{+,\tau}} [(\partial_{\nu,g} \beta) \overline{\beta_t} + (\overline{\partial_{\nu,g} \beta}) \beta_t] |g|^{1/2} \dif S \\
	\lesssim & \int_{\Sigma_{+,\tau}} |\partial_{\nu} \beta|^2 + |\beta_t|^2 \dif S_{g'}
	\lesssim \int_{\Sigma_{+,\tau}} e^{2\sigma \varphi} |\partial_{\nu} \alpha + \sigma \partial_{\nu} \varphi \cdot \alpha|^2 + e^{2\sigma \varphi} |\alpha_t + \sigma \varphi_t \alpha|^2 \dif S_{g'} \\
	\lesssim & \int_{\Sigma_{+,\tau}} e^{2\sigma \varphi} |\partial_{\nu} \alpha|^2 + e^{2\sigma \varphi} |\alpha_t|^2 + \sigma ^2 e^{2\sigma \varphi} |\alpha|^2 \dif S_{g'} \\
	= & \ \nrm[L^2(\Sigma_{+,\tau}, g')]{e^{\sigma \varphi} \partial_{\nu} \alpha}^2 + \nrm[L^2(\Sigma_{+,\tau}, g')]{e^{\sigma \varphi} \alpha_t}^2 + \sigma^2 \nrm[L^2(\Sigma_{+,\tau}, g')]{e^{\sigma \varphi} \alpha}^2.
	\end{align*}

	The positive-definiteness of the Riemannian metric gives $c |\nabla \omega|^2 \leq \agl[\df \omega,\df \omega] = 1 \leq C |\nabla \omega|^2$, so $0 < 1/C \leq |\nabla \omega|^2 \leq 1/c +\infty$.
	Combining \eqref{eq:3.2Temp1-ACR2019}, \eqref{eq:3.2Temp2-ACR2019}, \eqref{eq:3.2Temp3-ACR2019} and \eqref{eq:3.2Temp4-ACR2019} and integrating from $\tau = 1$ to $\tau = T$, we obtain
	\begin{align}
	& \int_{\Gamma_g} (\calX \beta) |g|^{1/2} \dif S \nonumber \\
	\lesssim & \ \int_{1}^T f_\sigma(\tau) \dif \tau + \nrm[L^2(Q_+;g)]{e^{\sigma \varphi} (\square_g + q) \alpha}^2 + \sigma \nrm[L^2(Q_+;g)]{e^{\sigma \varphi} \nabla_{x,t} \alpha}^2 + \sigma^3 \nrm[L^2(Q_+;g)]{e^{\sigma \varphi} \alpha}^2 \nonumber \\
	& + \nrm[L^2(\Sigma_{+})]{e^{\sigma \varphi} \partial_{\nu} \alpha}^2 + \nrm[L^2(\Sigma_{+})]{e^{\sigma \varphi} \alpha_t}^2 + \sigma^2 \nrm[L^2(\Sigma_{+})]{e^{\sigma \varphi} \alpha}^2 \nonumber\\
	\lesssim & \ \nrm[L^2(Q_+;g)]{e^{\sigma \varphi} (\square_g + q) \alpha}^2 + \sigma \nrm[L^2(Q_+;g)]{e^{\sigma \varphi} \nabla_{x,t} \alpha}^2 + \sigma^3 \nrm[L^2(Q_+;g)]{e^{\sigma \varphi} \alpha}^2 \nonumber\\
	& + \nrm[L^2(\Sigma_{+})]{e^{\sigma \varphi} \partial_{\nu} \alpha}^2 + \nrm[L^2(\Sigma_{+})]{e^{\sigma \varphi} \alpha_t}^2 + \sigma^2 \nrm[L^2(\Sigma_{+})]{e^{\sigma \varphi} \alpha}^2. \label{eq:energyBottomBeta-ACR2019}
	\end{align}
	Finally from \eqref{eq:energyBottomBeta-ACR2019} and \eqref{eq:alphabeta-ACR2019}, we arrive at
	\begin{align}
	& \int_{\Gamma_g} e^{2\sigma \varphi} (\calX \alpha) |g|^{1/2} \dif S \nonumber\\
	\leq & \ C \big[ \nrm[L^2(Q_+;g)]{e^{\sigma \varphi} (\square_g + q) \alpha}^2 + \sigma \nrm[L^2(Q_+;g)]{e^{\sigma \varphi} \nabla_{x,t} \alpha}^2 + \sigma^3 \nrm[L^2(Q_+;g)]{e^{\sigma \varphi} \alpha}^2 \nonumber\\
	& \ + \nrm[L^2(\Sigma_{+})]{e^{\sigma \varphi} \partial_{\nu} \alpha}^2 + \nrm[L^2(\Sigma_{+})]{e^{\sigma \varphi} \alpha_t}^2 + \sigma^2 \nrm[L^2(\Sigma_{+})]{e^{\sigma \varphi} \alpha}^2 \big], \label{eq:energyBottomNof-ACR2019}
	\end{align}
	which is \eqref{eq:energyBottom-ACR2019} with $f = 1$.

	In \eqref{eq:energyBottomNof-ACR2019}, \sq{changing} $\alpha$ to $f\alpha$, we obtain
	\begin{align}
	& \int_{\Gamma_g} e^{2\sigma \varphi} (\calX (f\alpha)) |g|^{1/2} \dif S \nonumber\\
	\leq & \ C \big[ \nrm[L^2(Q_+;g)]{e^{\sigma \varphi} (\square_g + q) (f\alpha)}^2 + \sigma \nrm[L^2(Q_+;g)]{e^{\sigma \varphi} \nabla_{x,t} (f\alpha)}^2 + \sigma^3 \nrm[L^2(Q_+;g)]{e^{\sigma \varphi} (f\alpha)}^2 \nonumber\\
	& \ + \nrm[L^2(\Sigma_{+})]{e^{\sigma \varphi} \partial_{\nu} (f\alpha)}^2 + \nrm[L^2(\Sigma_{+})]{e^{\sigma \varphi} (f\alpha)_t}^2 + \sigma^2 \nrm[L^2(\Sigma_{+})]{e^{\sigma \varphi} (f\alpha)}^2 \big]. \label{eq:energyBottom2f-ACR2019}
	\end{align}
	For $(\square_g + q) (f\alpha)$, we can compute
	\begin{align*}
		|(\square_g + q) (f\alpha)|
		& = |\square_g (f\alpha) + fq\alpha|
		= |(\square_g f)\alpha - 2 \agl[\df f,\df \alpha]+ 2 f_t \alpha_t + f\square_g \alpha + fq\alpha| \\
		& \leq |f\square_g \alpha| + 2 |\agl[\df f,\df \alpha]| + 2 |f_t| |\alpha_t| + |(\square_g f)\alpha| + |fq\alpha| \\
		& \leq \nrm[L^\infty(Q_+)]{f} |\square_g \alpha| + C_g |\nabla_{x,t} f| |\nabla_{x,t} \alpha| + \nrm[H^2(Q_+)]{f} |\alpha| + \nrm[L^\infty(Q_+)]{fq} |\alpha|,
	\end{align*}
	so
	\begin{align*}
		& \ \nrm[L^2(Q_+;g)]{e^{\sigma \varphi} (\square_g + q) (f\alpha)} \\
		\leq & \ C \big[ \nrm[L^2(Q_+;g)]{e^{\sigma \varphi} (\square_g + q) \alpha} + \nrm[L^2(Q_+;g)]{e^{\sigma \varphi} \nabla_{x,t} \alpha} + \nrm[L^2(Q_+;g)]{e^{\sigma \varphi} \alpha} \big].
	\end{align*}
	We can estimate $\nabla_{x,t} (f\alpha)$, $\partial_{\nu} (f\alpha)$ and $(f\alpha)_t$ in similar manners.
	Combining these estimates with \eqref{eq:energyBottom2f-ACR2019}, we finally arrive at \eqref{eq:energyBottom-ACR2019}.
\end{proof}

\section{Carleman estimates} \label{sec:CarlEst-ACR2019}

In order to obtain the results, certain Carleman estimates and strongly pseudoconvex functions are needed.
\sq{For the definition of the notion ``strongly pseudoconvex'', readers may refer to \cite{RSfixed2019}.}

First,we borrow a Carleman estimate with boundary term from \cite[Theorem A.7]{RSfixed2019}.
\sq{To make the paper self-contained, we present the statement of the result below, followed by the arguments that show how we modify the estimate to fit in our Riemannian settings.}
\begin{lem} \label{lem:CarlEst-ACR2019}
	Suppose $\Omega$ is a bounded open set in $\Rn$, $n \geq 2$, with a Lipschitz boundary, and $P(x, D)$ is a second order differential operator on $\overline \Omega$ with bounded coefficients whose principal symbol $p(x,\xi)$ has real $C^1$ coefficients. 
	If $\varphi$ is a smooth function on $\overline \Omega$ with $\nabla \varphi$ never zero in $\overline \Omega$ and $\varphi$ is strongly pseudoconvex with respect to $P(x, D)$ on $\overline \Omega$, then for large enough $\sigma$ and for all $u \in C^2(\overline \Omega)$ one has
	\begin{align*}
		\sigma \int_\Omega e^{2\sigma \varphi} (|\nabla v|^2 + \sigma^2 |v|^2) \dif x + \sigma \int_{\partial \Omega} \nu^j E_j \dif S 
		\leq \int_\Omega e^{2\sigma \varphi} |P v|^2 \dif x
	\end{align*}
	with the constant independent of $\sigma$ and u.
	Here $v = (v_1, \dots,v_n)$ is the outward unit normal to $\partial \Omega$, 
	\begin{equation*}
		E_j = \nu^j A(x,\nabla \alpha,\sigma \alpha) \frac {\partial B} {\partial \xi_j}(x) - \frac {\partial A} {\partial \xi_j} (x,\nabla_x \alpha,\sigma \alpha) [B(x,\nabla \alpha) + \alpha \mathbb G],
	\end{equation*}
	$\alpha = e^{\sigma \varphi} v$, $\mathbb G$ some real valued function independent of $\lambda$, $\sigma$, $v$ and
	\begin{equation*}
	A(x,\xi,\sigma) = p(x,\xi) - \sigma^2 p(x,\nabla \varphi),\quad 
	B(x,\xi)
	= \{p,\varphi\}(x,\xi).
	\end{equation*}
\end{lem}

\sq{A proof of Lemma \ref{lem:CarlEst-ACR2019} can be find in \cite[Theorem A.7]{RSfixed2019}.}

\begin{rem}
	The $E_j$ in the lemma are constructed in \cite[Theorem A.7]{RSfixed2019} and would seem to depend on the domains $Q_\pm$. 
	However, the $E_j$ depend on $\mathbb G$ which itself depends on a function $h$ which satisfies the algebraic identity \cite[eq.~(A.25)]{RSfixed2019}. 
	We can construct the $h$ so that the algebraic identity is satisfied on $Q$ rather than $Q_+$ and $Q_-$ separately. 
	Then the $\mathbb G$ in Lemma \ref{lem:CarlEst-ACR2019} will be the same for $Q_+$ and $Q_-$.
\end{rem}

In our setting, the corresponding ``$\Omega$'' is $Q_+$, 
the outer normal $\nu$ of $\Gamma_g$ pointing into $Q_-$ is $\nu(x) = (\nabla \omega(x), -1) / \sqrt{|\nabla \omega(x)|^2 + 1}$, 
the operator $\square_g$ we are interested in is 
\begin{equation} \label{thm:Operf-ACR2019}
	(\square_g + q) f \ \colon v \ \mapsto \ (\square_g + q) (fv)
\end{equation}
where $f \in C^2(\overline{Q_+})$ is nonzero,
the principal symbol is $p(x,t;\xi,\tau) = f(x,t)[-\tau^2 + g^{jk}(x) \xi_j \xi_k]$,
and the corresponding $A$ and $B$ of Lemma \ref{lem:CarlEst-ACR2019} are
\begin{align*}
A(x,t;\xi,\tau;\sigma) 
& = f(x,t)[-\tau^2 + g^{jk} \xi_j \xi_k - \sigma^2 (-\varphi_t^2 + g^{jk} \varphi_j \varphi_k)], \\
B(x,t;\xi,\tau)
& = f(x,t)(2g^{jk} \xi_k, -2\tau) \cdot (\varphi_j, \varphi_t) = f(x,t)(2g^{jk} \xi_k \varphi_j - 2\tau \varphi_t).
\end{align*}
Now we show that the $\nu^j E_j$ in Lemma \ref{lem:CarlEst-ACR2019} equals $e^{\sigma \varphi} F(\sigma v,X^1 v, \cdots,X^n v)$ for some quadratic form $F$ with smooth coefficients depending on $x$ and $t$.
According to Lemma \ref{lem:CarlEst-ACR2019} and Lemma \ref{lem:Claim1Equ-ACR2019}, we can compute the corresponding $\nu^j E_j$ as
\begin{align*}
& \ \sqrt{|\nabla \omega(x)|^2 + 1} (\nu^j E_j)/f^2 \\
= & \ \nu^j A(x,t;\nabla_x \alpha,\alpha_t;\sigma \alpha) \frac {\partial B} {\partial \xi_j}(x,t) - \nu^j \frac {\partial A} {\partial \xi_j} (x,t;\nabla_x \alpha,\alpha_t;\sigma \alpha) [B(x,t;\nabla_x \alpha, \alpha_t) + g \alpha] \\
& \ + \nu^{n+1} A(x,t;\nabla_x \alpha,\alpha_t;\sigma \alpha) \frac {\partial B} {\partial \tau}(x,t) \\
& \ - \nu^{n+1} \frac {\partial A} {\partial \tau} (x,t;\nabla_x \alpha,\alpha_t;\sigma \alpha) [B(x,t;\nabla_x \alpha, \alpha_t) + g \alpha] \\
= & \ \omega_j [-\alpha_t^2 + g^{jk} \alpha_j \alpha_k - \sigma^2 \alpha^2 (g^{jk} \varphi_j \varphi_k - \varphi_t^2)] (2g^{jk} \varphi_k) \\
& \ - \omega_j 2g^{jk} \alpha_k (2g^{jk} \alpha_k \varphi_j - 2\alpha_t \varphi_t + g\alpha) \\
& \ + (-1) [-\alpha_t^2 + g^{jk} \alpha_j \alpha_k - \sigma^2 \alpha^2 (g^{jk} \varphi_j \varphi_k - \varphi_t^2)] (-2\varphi_t) \\
& \ - (-1) (-2\alpha_t) (2g^{jk} \alpha_k \varphi_j - 2\alpha_t \varphi_t + g\alpha) \\
= & \ 2(X \varphi) [-\alpha_t^2 + g^{jk} \alpha_j \alpha_k - \sigma^2 \alpha^2 (g^{jk} \varphi_j \varphi_k - \varphi_t^2)] - 2(X \alpha) (2g^{jk} \alpha_k \varphi_j - 2\alpha_t \varphi_t + g\alpha) \\
= & \ 2(X \varphi) [\sum_m (X^m \alpha) (X^m \alpha) - 2\alpha_t (X\alpha) - \sigma^2 \alpha^2 (g^{jk} \varphi_j \varphi_k - \varphi_t^2)] \\
& \ - 4(X \alpha) [\sum_m (X^m \alpha) (X^m \varphi) - \alpha_t (X\varphi) - \varphi_t (X\alpha) + g\alpha/2] \\
= & \ 2(X \varphi) [\sum_m (X^m \alpha)^2 - \sigma^2 \alpha^2 (g^{jk} \varphi_j \varphi_k - \varphi_t^2)] - 4(X \alpha) [\sum_m (X^m \alpha) (X^m \varphi) - \varphi_t (X\alpha) + g\alpha/2] \\
& = \tilde F(\sigma \alpha,X^1 \alpha, \cdots,X^n \alpha),
\end{align*}
where $\tilde F$ is some quadratic form in its variables with smooth coefficients depending on $x$ and $t$.
Due to $\alpha = e^{\sigma \varphi} v$, and noticing that $X \alpha = e^{\sigma \varphi} (X v + (X\varphi) \sigma v)$, we can find another form $F$ such that
\begin{equation} \label{eq:FQua-ACR2019}
	\nu^j E_j = \frac {f^2(x,t)\tilde F(\sigma \alpha,X^1 \alpha, \cdots,X^n \alpha)} {\sqrt{|\nabla \omega(x)|^2 + 1}} = e^{2\sigma \varphi} F(\sigma v,X^1 v, \cdots,X^n v).
\end{equation}
Here, the $F$ is the desired quadratic form with smooth coefficients depending on $x$ and $t$.
Later, we shall represent $F(\sigma v,X^1 v, \cdots,X^n v)$ by $F(\sigma v,X^m v)$ for simplicity.
One should pay special attention to the fact that \sq{the} differential operators $X^m$ in \eqref{eq:FQua-ACR2019} are all tangential to $\Gamma_g$, and later we shall \sq{control} $\int_{\Gamma_g} \nu^j E_j \dif S$ \sq{with} norms of \sq{certain} one-forms over the hypersurface $\Gamma_g$.

The following result, which is a Riemannian counterpart of \cite[Lemma 3.1]{RSfixed2019March}, constructs a suitable strongly pseudoconvex function for $\square_g$ starting from a strictly convex function in $(\overline{B},g)$.

\begin{lem} \label{lem:weight-ACR2019}
	Assume $\tilde \psi \in C^2(\overline B)$ \sq{has no critical points in $\overline B$} and is strictly convex for all $x \in \overline B$ in terms of the metric $g$, i.e.~the Hessian matrix $\tilde \psi_{jk} - \tilde \psi_i \Gamma^i_{jk}$ is positive definite.
Let $\psi(x) = \tilde \psi(x) + (\sup_B \tilde \psi - 2 \inf_B \tilde \psi)$ and let $\phi(x,t) = e^{\iota \psi(x)} - (t - \omega(x))^2/2$.
Then when $\iota$ and $\lambda$ are large enough, $\varphi = e^{\lambda \phi}$ is strongly pseudoconvex for $\square_q$ near $\overline{Q}$.
\sq{Moreover, there exists a number $T$ (given by \eqref{eq:vofT-ACR2019}), depending on $\iota$, $\psi$ and $\omega$, such that}
the smallest value of $\varphi$ on $\Gamma_g$ is strictly larger than the largest value of $\varphi$ on $\Gamma_T \cup \Gamma_{-T}$, i.e.,
\begin{equation} \label{eq:varphiSmallBig-ACR2019}
\inf\limits_{\Gamma_g} \varphi > \sup\limits_{\Gamma_T \cup \Gamma_{-T}} \varphi,
\end{equation}
and the function
\begin{equation} \label{eq:defH-ACR2019}
h(\sigma) := \sup_{x \in \overline B} \int_{-T}^T e^{2\sigma[\varphi(x,t)-\varphi(x,\omega(x))]} \dif t
\end{equation}
satisfies $\lim_{\sigma \to +\infty} h(\sigma) = 0$.
\end{lem}
\begin{proof}
	\sq{First, we check that $|\nabla_{(x,t)} \phi| \neq 0$ in $\overline{Q^+}$.
	The time derivative $\partial_t \phi = \omega(x) - t$, so when $(x,t) \in \overline{Q^+} \backslash \Gamma_g$, we have $\partial_t \phi \neq 0$, so $|\nabla_{(x,t)} \phi| \neq 0$ in $\overline{Q^+}$.
	On $\Gamma_g$ we know $\omega(x) - t = 0$, so $\nabla_x \phi = \iota e^{\iota \psi} \nabla \psi - (\omega(x) - t) \nabla \omega = \iota e^{\iota \psi} \nabla \psi(x)$ holds on $\Gamma_g$.
	Because $\tilde \psi$ (and so $\psi$) has no critical points in $\overline B$, we conclude $|\nabla_x \phi| \neq 0$ on $\Gamma_g$.
	In total, $|\nabla_{(x,t)} \phi| \neq 0$ in $\overline{Q^+}$.}

	According to Proposition A.3 and A.5 in \cite{RSfixed2019}, we only need to show that the level surfaces of $\phi$ are pseudoconvex in $Q_+$. 
	That is to say, we need to show
	\[
	\{p, \{p, \phi\}\}(x,t; \xi,\tau) > 0
	\]
	whenever
	\[
	(x,t) \in Q_+, \, (\xi,\tau) \in \R^{n+1} \backslash \{0\}, \, p(x,t; \xi,\tau) = 0, \, \{p, \phi\}(x,t; \xi,\tau) = 0,
	\]
	where $p(x,t; \xi, \tau) := g^{jk} \xi_j \xi_k - \tau^2$ is the principal symbol of $\square_g$.
	\smallskip

	From $p = 0$ we obtain $g^{jk} \xi_j \xi_k = \tau^2$.
	Denote the integral curve of the Hamiltonian of $p$ as $(x(s),t(s),\xi(s),\tau(s))$,
	which satisfy $\dot x^j(s) = 2 g^{jk} \xi_k$ and $\dot t(s) = -2\tau$.
	Hence from $\{p,\phi\} = 0$ we obtain $\iota \psi_i \dot x^i = e^{-\iota \psi} (t - \omega) (\dot t - \omega_i \dot x^i)$, so
	\begin{equation*}
	(\dot t)^2 = 4\tau^2 = 4 g^{\ell m} \xi_\ell \xi_m = 4 g^{\ell j} g_{j k} g^{km} \xi_\ell \xi_m = (2 g^{\ell j} \xi_\ell) g_{jk} (2 g^{km} \xi_m) = g_{jk} \dot x^j \dot x^k.
	\end{equation*}
	It is known that the projection of this integral curve, i.e.~$(x(s),t(s))$, is a geodesic in terms of $-(\df t)^2 + g$, so
	\begin{equation*}
	\ddot x^i + \Gamma^i_{jk} \dot x^j \dot x^k = 0,
	\quad\text{and}\quad
	\ddot t = 0.
	\end{equation*}
	Combining these analysis, we can compute $\{p,\{p,\phi\}\}$ as follows,
	\begin{align}
		\{p,\{p,\phi\}\}
		& = \frac {\df^2} {\df s^2} \big( \phi(x(s),t(s)) \big)
		= \phi_i \ddot x^i + \phi_{jk} \dot x^j \dot x^k + \phi_{tt} (\dot t)^2 + \phi_t \ddot t \nonumber \\
		& = (\phi_{jk} - \phi_i \Gamma^i_{jk}) \dot x^j \dot x^k - (\dot t)^2 \nonumber \\
		& = \big\{ [(\iota \psi_{jk} + \iota^2 \psi_j \psi_k) e^{\iota \psi} + \omega_{jk} (t - \omega) - \omega_j \omega_k] - [\iota \psi_i e^{\iota \psi} + \omega_i (t - \omega)] \Gamma^i_{jk} \big\} \dot x^j \dot x^k \nonumber \\
		& \quad - (\dot t)^2 \nonumber \\
		& = \big[ \iota e^{\iota \psi} (\psi_{jk} - \psi_i \Gamma^i_{jk}) + (t - \omega) (\omega_{jk} - \omega_i \Gamma^i_{jk}) - \omega_j \omega_k \big] \dot x^j \dot x^k - (\dot t)^2 \nonumber \\
		& \quad + e^{-\psi}  (t - \omega)^2 (\dot t - \omega_i \dot x^i)^2 \nonumber \\
		& \geq \big[ \iota e^{\iota \psi} (\psi_{jk} - \psi_i \Gamma^i_{jk}) + (t - \omega) (\omega_{jk} - \omega_i \Gamma^i_{jk}) - \omega_j \omega_k \big] \dot x^j \dot x^k - (\dot t)^2 \nonumber \\
		& = \big[ \iota e^{\iota \psi} (\psi_{jk} - \psi_i \Gamma^i_{jk}) + (t - \omega) (\omega_{jk} - \omega_i \Gamma^i_{jk}) - \omega_j \omega_k - g_{jk} \big] \dot x^j \dot x^k \nonumber \\
		& \geq e^{\iota \psi} \big[ \iota (\psi_{jk} - \psi_i \Gamma^i_{jk}) - e^{-\iota \psi} |T \pm \omega| |\omega_{jk} - \omega_i \Gamma^i_{jk}| - e^{-\iota \psi} (\omega_j \omega_k + g_{jk}) \big] \dot x^j \dot x^k. \nonumber 
	\end{align}
Note that the condition $\tilde \psi_{jk} - \tilde \psi_i \Gamma^i_{jk}$ being positive definite implies $\psi_{jk} - \psi_i \Gamma^i_{jk}$ is also positive definite.
We shall set a value for $T$ such that $e^{-\iota \psi(x)} |T \pm \omega(x)|$ is uniformly bounded (in terms of $x$) by a constant, e.g.,
\begin{equation} \label{eq:Tw2-ACR2019}
e^{-\iota \psi(x)} |T \pm \omega(x)| < 2,
\end{equation}
for the purpose of guaranteeing $\{p,\{p,\phi\}\} > 0$ by taking $\iota$ to be large enough.
But before that, we first investigate the requirement \eqref{eq:varphiSmallBig-ACR2019}.
\sq{We set
\begin{equation} \label{eq:vofT-ACR2019}
T = T(\iota, \psi, \omega) := \sup_{x \in \overline B} \big( e^{\iota \psi(x)} - \inf\limits_{y \in \overline B} e^{\iota \psi(y)} \big)^{1/2} + \sup_{x \in \overline B} |\omega(x)| + 1,
\end{equation}
then $T - |\omega(x)| > 0$ and so $0 < T - |\omega(x)| \leq T \pm \omega(x)$.
Hence,
\begin{align}
\inf_{x \in \overline B} (T \pm \omega(x))^2
& \geq \inf_{x \in \overline B} (T - |\omega(x)|)^2
= (T - \sup_{x \in \overline B} |\omega(x)|)^2 \nonumber \\
& = \big[ \sup_{x \in \overline B} \big( e^{\iota \psi(x)} - \inf_{y \in \overline B} e^{\iota \psi(y)} \big)^{1/2} + 1 \big]^2 \nonumber \\
& > \sup_{x \in \overline B} e^{\iota \psi(x)} - \inf\limits_{y \in \overline B} e^{\iota \psi(y)}. \label{eq:voT-ACR2019}
\end{align}
Therefore, one can have
\begin{align*}
	\sup\limits_{\Gamma_T \cup \Gamma_{-T}} \phi
	& = \sup_{x \in \overline B} [e^{\iota \psi(x)} - (T \pm \omega(x))^2]
	\leq \sup_{x \in \overline B} e^{\iota \psi(x)} - \inf_{x \in \overline B} (T \pm \omega(x))^2 \\
	& < \sup_{x \in \overline B} e^{\iota \psi(x)} - (\sup_{x \in \overline B} e^{\iota \psi(x)} - \inf\limits_{y \in \overline B} e^{\iota \psi(y)}) \qquad (\text{by~} \eqref{eq:voT-ACR2019})\\
	& = \inf\limits_{y \in \overline B} e^{\iota \psi(y)}
	= \inf\limits_{(y,t) \in \Gamma_g} [e^{\iota \psi(y)} - (t - \omega(y)^2/2)] \\
	& = \inf\limits_{\Gamma_g} \phi,
\end{align*}
which implies \eqref{eq:varphiSmallBig-ACR2019} since $\varphi = e^{\lambda \phi}$ and $\lambda$ is a positive number.}

\sq{
Now we show \eqref{eq:defH-ACR2019}.
We have
\begin{align*}
	\varphi(x,t)-\varphi(x,\omega(x))
	& = e^{\lambda e^{\iota \psi(x)}} [e^{-\lambda (t - \omega(x))^2/2} - 1]
	\leq e^{-\lambda (t - \omega(x))^2/2} - 1.
\end{align*}
The last inequality is due to $e^{\lambda e^{\iota \psi(x)}} \geq 1$ and $e^{-\lambda (t - \omega(x))^2/2} - 1 \leq 0$.
It can be checked that for $t \in \R$,
\begin{equation} \label{eq:et2-ACR2019}
e^{-t^2/2} - 1 \leq
\begin{cases}
	- \frac 1 2 e^{-1/2} t^2, & |t| < 1, \\
	\frac 1 2 e^{-1/2} - 1 \leq - \frac 1 2, & |t| \geq 1.
\end{cases}
\end{equation}
Hence,
\begin{align*}
h(\sigma)
& \leq \sup_{x \in \overline B} \int_{-T}^{T} e^{2\sigma (e^{-\lambda (t - \omega(x))^2/2} - 1)} \dif t 
= \sup_{x \in \overline B} \int_{-T-\omega(x)}^{T-\omega(x)} e^{2\sigma (e^{-\lambda t^2/2} - 1)} \dif t \\
& \leq \sup_{x \in \overline B} \int_{-T-\omega(x)}^{T-\omega(x)} e^{2\sigma (e^{-t^2/2} - 1)} \dif t \qquad (\because \lambda > 1) \\
& \leq \int_{-1}^{1} e^{-\sigma e^{-1/2} t^2} \dif t + \int_{-T-\sup_{x \in \overline B} |\omega(x)|}^{T + \sup_{x \in \overline B} |\omega(x)|} e^{-\sigma} \dif t \qquad (\text{by~} \eqref{eq:et2-ACR2019}) \\
& \leq \int_{-\infty}^{+\infty} e^{-\sigma e^{-1/2} t^2} \dif t  + 2(T + \sup_{x \in \overline B} |\omega(x)|) e^{-\sigma} \\
& = \sqrt \pi e^{1/4} \sigma^{-1/2} + 2(T + \sup_{x \in \overline B} |\omega(x)|) e^{-\sigma}.
\end{align*}
Therefore, we have \eqref{eq:defH-ACR2019}.}

Now we go back to \eqref{eq:Tw2-ACR2019}.
Recall
\[
\psi(x)
= \tilde \psi(x) + (\sup_B \tilde \psi - 2 \inf_B \tilde \psi)
= (\tilde \psi(x) - \inf_B \tilde \psi) + (\sup_B \tilde \psi - \inf_B \tilde \psi),
\]
which implies $\sup_B \psi \leq 2 \inf_B \psi$.
Combining this with \eqref{eq:vofT-ACR2019}, we can estimate $e^{-\iota \psi(x)} |T \pm \omega(x)|$ as follows,
\begin{align*}
& \ e^{-\iota \psi(x)} |T \pm \omega(x)| \\
= & \ e^{-\iota \psi(x)} |\sup_{x \in \overline B} \big( e^{\iota \psi(x)} - \inf\limits_{y \in \overline B} e^{\iota \psi(y)} \big)^{1/2} + \sup_{x \in \overline B} |\omega(x)| + 1 \pm \omega(x)| \\
\leq & \ \sup_{z \in \overline B} \big( e^{\iota (\psi(z) - 2 \psi(x))} - \inf\limits_{y \in \overline B} e^{\iota (\psi(y) - 2 \psi(x))} \big)^{1/2} + e^{-\iota \psi(x)} (\sup_{\overline B} |\omega| - \inf_{\overline B} |\omega| + 1) \\
\leq & \ 1 + e^{-\iota \psi(x)} (\sup_{\overline B} |\omega| - \inf_{\overline B} |\omega| + 1) \\
\leq & \ 2.
\end{align*}
The last inequality holds when $\iota$ is large enough.
Formula \eqref{eq:Tw2-ACR2019} is proved.
By the positive-definiteness of $\psi_{jk} - \psi_i \Gamma^i_{jk}$, we can conclude $\{p,\{p,\phi\}\}$ is strictly greater than 0 when $\iota$ is large enough.
The proof is complete.
\end{proof}

The following proposition is analogous to \cite[Proposition 3.2]{RSfixed2019} and will be used in the solution of the inverse problem.

\begin{prop} \label{prop:vpmAsm-ACR2019}
	Assume that $g$ satisfies the assumption \eqref{ass2-ACR2019}, and that $v_{\pm} \in H^2(Q_{\pm})$ satisfy 
\begin{equation} \label{eq:vpmAsm-ACR2019}
	(\square_g + q_{1,\pm}) v_\pm = -X \left( v_\pm / a_{-1,\pm} \right)|_{(x,\omega(x))} \cdot u_{2,\pm} \quad \text{in} \quad Q_\pm,
	\end{equation}
	where $a_{-1,\pm} \in C^{\infty}(\Gamma_g)$ and $\norm{q_{1,\pm}}_{L^{\infty}(B)} \leq M$, $\norm{u_{2,\pm}}_{L^{\infty}(Q_{\pm})} \leq M$ for some $M > 0$. Then, if $T$ and $\sigma$ are large enough but fixed, one has 
	\begin{align*}
	& \sum_\pm \int_{\Gamma_g} \calX (v_\pm / a_{-1,\pm}) \, |g|^{1/2} \dif S \\
	\lesssim & \ \int_{\Gamma_g} \calX (v_+ / a_{-1,+} - v_- / a_{-1,-}) \dif S + \sum_\pm (\nrm[H^1(\Sigma_{\pm};g')]{v_\pm}^2 + \nrm[L^2(\Sigma_{\pm};g')]{\partial_{\nu,g} v_\pm}^2)
	\end{align*}
	with constants depending on $M$.
\end{prop}

\begin{proof}[Proof of Proposition \ref{prop:vpmAsm-ACR2019}]
	The assumption \eqref{eq:vpmAsm-ACR2019} implies
	\begin{align}
		\nrm[L^2(Q_\pm)]{(\square_g + q_{1,\pm}) v_\pm}
		& = \nrm[L^2(Q_\pm)]{-X |_{\gamma(s)} (v_\pm / a_{-1,\pm}) \cdot u_{2,\pm}} \nonumber\\
		& \leq \nrm[L^\infty(Q_+)]{u_{2,\pm}} \cdot \nrm[L^2(Q_\pm)]{X |_{\gamma(s)} (v_\pm / a_{-1,\pm})} \nonumber\\
		& \lesssim \nrm[L^2(\Gamma_g)]{X (v_\pm / a_{-1,\pm})} \nonumber\\
		& \lesssim \sum_m \nrm[L^2(\Gamma_g)]{X^m (v_\pm/a_{-1,\pm})}. \label{eq:pm4-ACR2019}
	\end{align}
	Likewise, we have
	\begin{align}
	\nrm[L^2(Q_\pm)]{e^{\sigma \varphi} (\square_g + q_{1,\pm}) v_\pm}^2
	& \lesssim h(\sigma) \sum_m \nrm[L^2(\Gamma_g)]{e^{\sigma \varphi} X^m (v_\pm / a_{-1,\pm})}^2, \label{eq:pm6-ACR2019}
	\end{align}
	where $\varphi$ is strongly pseudoconvex for $\square_g$ near $Q$, 
	and $h(\sigma) := \sup\limits_{x \in B} \int_{-T}^{T} e^{2\sigma [\varphi(x,t) - \varphi(x, \omega(x))]} \dif t$.
	Thanks to the assumption \eqref{ass2-ACR2019}, Lemma \ref{lem:weight-ACR2019} can be applied and we can find such a strongly pseudoconvex function $\varphi$.
	
	Combining \eqref{eq:pm4-ACR2019} and Lemma \ref{lem:energyTop-ACR2019}, we can have
	\begin{align*}
	& \int_{\Gamma_T} \big( (\df (v / a_{-1}), \overline{\df (v / a_{-1})})_g + |(v / a_{-1})_t|^2 + \sigma^2 |v / a_{-1}|^2 \big) |g|^{1/2} \dif S \\
	\lesssim & \ \int_{\Gamma_g} (\calX (v / a_{-1})) |g|^{1/2} \dif S + \nrm[H^1(\Sigma_{+};g')]{v / a_{-1}}^2 + \nrm[L^2(\Sigma_{+};g')]{\partial_{\nu,g} (v / a_{-1})}^2.
	\end{align*}
	If we take into consideration the term $e^{\sigma \varphi}$, we can continue
	\begin{align}
	& \int_{\Gamma_T} e^{2\sigma \varphi} \big( (\df (v / a_{-1}), \overline{\df (v / a_{-1})})_g + |(v / a_{-1})_t|^2 + \sigma^2 |v / a_{-1}|^2 \big) |g|^{1/2} \dif S \nonumber\\
	\lesssim & \ e^{-2\sigma\delta} \int_{\Gamma_g} e^{2\sigma \varphi} (\calX (v / a_{-1})) |g|^{1/2} \dif S + e^{C\sigma} (\nrm[H^1(\Sigma_{+};g')]{v / a_{-1}}^2 + \nrm[L^2(\Sigma_{+};g')]{\partial_{\nu,g} (v / a_{-1})}^2), \label{eq:pm5-ACR2019}
	\end{align}
	where the $\delta := \sup\limits_{(x,t) \in \Gamma_T} \varphi(x,t) - \inf\limits_{(x,t) \in \Gamma_g} \varphi(x,t) < 0$ and the constant $C = 2\sup\limits_{x \in \Gamma_T} \varphi(x)$.

	Lemma \ref{lem:energyBottom-ACR2019} gives
	\begin{align}
	& \int_{\Gamma_g} e^{2\sigma \varphi} (\calX (v/a_{-1})) |g|^{1/2} \dif S - C \sigma \nrm[L^2(Q_+;g)]{e^{\sigma \varphi} \nabla_{x,t} v}^2 - C \sigma^3 \nrm[L^2(Q_+;g)]{e^{\sigma \varphi} v}^2\nonumber\\
	\lesssim & \ C (\nrm[L^2(Q_+;g)]{e^{\sigma \varphi} (\square_g + q) v}^2 + \nrm[L^2(\Sigma_{+})]{\partial_{\nu} v}^2 + \nrm[L^2(\Sigma_{+})]{e^{\sigma \varphi} v_t}^2 + \sigma^2 \nrm[L^2(\Sigma_{+})]{e^{\sigma \varphi} v}^2) \label{eq:pm2-ACR2019}
	\end{align}
	for some constant $C$.

By Lemma \ref{lem:CarlEst-ACR2019} and \eqref{eq:FQua-ACR2019}, for sufficiently large $\sigma$ we have
	\begin{align} \label{eq:pm1-ACR2019}
	& \sigma \nrm[L^2(Q_+)]{e^{\sigma \varphi} \nabla_{x,t} v}^2 + \sigma^3 \nrm[L^2(Q_+)]{e^{\sigma \varphi} v}^2 + \sigma \int_{\partial Q_+} e^{2\sigma \varphi} F(\sigma (\frac v {a_{-1}}),X^m (\frac v {a_{-1}})) \dif S \nonumber\\
	\lesssim & \ \sigma \nrm[L^2(Q_+)]{e^{\sigma \varphi} \nabla_{x,t} (\frac v {a_{-1}})}^2 + \sigma^3 \nrm[L^2(Q_+)]{e^{\sigma \varphi} (\frac v {a_{-1}})}^2 + \sigma \int_{\partial Q_+} e^{2\sigma \varphi} F(\sigma (\frac v {a_{-1}}),X^m (\frac v {a_{-1}})) \dif S \nonumber\\
	\lesssim & \ \nrm[L^2(Q_+)]{e^{\sigma \varphi} \mathcal L(\frac v {a_{-1}})}^2
	= \nrm[L^2(Q_+)]{e^{\sigma \varphi} (\square_g + q)v}^2,
	\end{align}
	where $\mathcal L = (\square_g + q)a_{-1}$ is as the operator \eqref{thm:Operf-ACR2019} where we set $f$ to be $a_{-1}$.

	Summing up \eqref{eq:pm1-ACR2019} and \eqref{eq:pm2-ACR2019} and dividing the boundary integral on $\partial Q_+$ into integrals on $\Gamma_g$ and on $\Sigma_+ \cup \Gamma_T$, and combining with \eqref{eq:pm6-ACR2019}, we obtain
	\begin{align}
		& \int_{\Gamma_g} e^{2\sigma \varphi} (\calX (v/a_{-1})) |g|^{1/2} \dif S + \sigma \int_{\Gamma_g} e^{2\sigma \varphi} F(\sigma (v / a_{-1}),X^m (v / a_{-1})) \dif S \nonumber\\
		\lesssim & \ \nrm[L^2(Q_+;g)]{e^{\sigma \varphi} (\square_g + q) v}^2 + \nrm[L^2(\Sigma_{+})]{\partial_{\nu} v}^2 + \nrm[L^2(\Sigma_{+})]{e^{\sigma \varphi} v_t}^2 + \sigma^2 \nrm[L^2(\Sigma_{+})]{e^{\sigma \varphi} v}^2 \nonumber\\
		& \ + \nrm[L^2(\Gamma_T)]{e^{\sigma \varphi} v_t}^2 + \sigma^2 \nrm[L^2(\Gamma_T)]{e^{\sigma \varphi} v}^2 \nonumber\\
		\lesssim & \ h(\sigma) \sum_m \nrm[L^2(\Gamma_g)]{e^{\sigma \varphi} X^m (v / a_{-1})}^2 + \nrm[L^2(\Sigma_{+})]{\partial_{\nu} v}^2 + \nrm[L^2(\Sigma_{+})]{e^{\sigma \varphi} v_t}^2 + \sigma^2 \nrm[L^2(\Sigma_{+})]{e^{\sigma \varphi} v}^2 \nonumber\\
		& \ + \nrm[L^2(\Gamma_T)]{e^{\sigma \varphi} v_t}^2 + \sigma^2 \nrm[L^2(\Gamma_T)]{e^{\sigma \varphi} v}^2 \nonumber\\
		\lesssim & \ h(\sigma) \nrm[L^2(\Gamma_g)]{e^{\sigma \varphi} \calX (v / a_{-1})}^2 + \nrm[L^2(\Sigma_{+})]{\partial_{\nu} v}^2 + \nrm[L^2(\Sigma_{+})]{e^{\sigma \varphi} v_t}^2 + \sigma^2 \nrm[L^2(\Sigma_{+})]{e^{\sigma \varphi} v}^2 \nonumber\\
		& \ + \nrm[L^2(\Gamma_T)]{e^{\sigma \varphi} v_t}^2 + \sigma^2 \nrm[L^2(\Gamma_T)]{e^{\sigma \varphi} v}^2. \label{eq:pm3-ACR2019}
	\end{align}
	According to Lemma \ref{lem:weight-ACR2019}, as $\sigma$ goes to infinity, $h(\sigma)$ goes to zero, so the term containing $h(\sigma)$ on the right-hand-side of \eqref{eq:pm3-ACR2019} can be absorbed by the corresponding term on the left-hand-side, and then gives us
	\begin{align}
	& \int_{\Gamma_g} e^{2\sigma \varphi} (\calX (v / a_{-1})) |g|^{1/2} \dif S + \sigma \int_{\Gamma_g} e^{2\sigma \varphi} F(\sigma (v / a_{-1}),X^m (v / a_{-1})) \dif S \nonumber\\
	\lesssim & \ \nrm[L^2(\Sigma_{+})]{\partial_{\nu} v}^2 + \nrm[L^2(\Sigma_{+})]{e^{\sigma \varphi} v_t}^2 + \sigma^2 \nrm[L^2(\Sigma_{+})]{e^{\sigma \varphi} v}^2 + \nrm[L^2(\Gamma_T)]{e^{\sigma \varphi} v_t}^2 + \sigma^2 \nrm[L^2(\Gamma_T)]{e^{\sigma \varphi} v}^2. \label{eq:pm7-ACR2019}
	\end{align}
	Combining \eqref{eq:pm5-ACR2019} and \eqref{eq:pm7-ACR2019}, we obtain
	\begin{align*}
	& \int_{\Gamma_g} e^{2\sigma \varphi} (\calX (v / a_{-1})) |g|^{1/2} \dif S + \sigma \int_{\Gamma_g} e^{2\sigma \varphi} F(\sigma (v / a_{-1}),X^m (v / a_{-1})) \dif S \\
	\lesssim & \ \nrm[L^2(\Sigma_{+})]{\partial_{\nu} v}^2 + \nrm[L^2(\Sigma_{+})]{e^{\sigma \varphi} v_t}^2 + \sigma^2 \nrm[L^2(\Sigma_{+})]{e^{\sigma \varphi} v}^2 \\
	& \ + e^{-2\sigma\delta} \int_{\Gamma_g} e^{2\sigma \varphi} (\calX v) |g|^{1/2} \dif S + e^{C\sigma} (\nrm[H^1(\Sigma_{+};g')]{v}^2 + \nrm[L^2(\Sigma_{+};g')]{\partial_{\nu,g} v}^2).
	\end{align*}
	Again, the term containing $e^{-2\sigma\delta}$ on the right-hand-side can be absorbed to the left-hand-side when $\sigma$ is large enough, so
	\begin{align}
	& \int_{\Gamma_g} e^{2\sigma \varphi} (\calX (v_- / a_{-1,-})) |g|^{1/2} \dif S + \sigma \int_{\Gamma_g} e^{2\sigma \varphi} F(\sigma (v / a_{-1}),X^m (v / a_{-1})) \dif S \nonumber\\
	\lesssim & \ \nrm[L^2(\Sigma_{+})]{\partial_{\nu} v}^2 + \nrm[L^2(\Sigma_{+})]{e^{\sigma \varphi} v_t}^2 + \sigma^2 \nrm[L^2(\Sigma_{+})]{e^{\sigma \varphi} v}^2 \nonumber\\
	& \ + e^{C\sigma} (\nrm[H^1(\Sigma_{+};g')]{v}^2 + \nrm[L^2(\Sigma_{+};g')]{\partial_{\nu,g} v}^2) \nonumber\\
	\lesssim & \ \sigma^2 e^{C\sigma} (\nrm[H^1(\Sigma_{+};g')]{v}^2 + \nrm[L^2(\Sigma_{+};g')]{\partial_{\nu,g} v}^2), \label{eq:pm8-ACR2019}
	\end{align}
	where the constant $C = 2\sup \{ \varphi(x) \,;\, x \in \Gamma_T \cup \Sigma_+ \}$.

	Similarly, for $v_-$ we have
	\begin{align}
	& \int_{\Gamma_g} e^{2\sigma \varphi} (\calX (v_- / a_{-1,-})) |g|^{1/2} \dif S - \sigma \int_{\Gamma_g} e^{2\sigma \varphi} F(\sigma (v_- / a_{-1,-}),X^m (v_- / a_{-1,-})) \dif S \nonumber\\
	\lesssim & \ \sigma^2 e^{C\sigma} (\nrm[H^1(\Sigma_{-};g')]{v_-}^2 + \nrm[L^2(\Sigma_{-};g')]{\partial_{\nu,g} v_-}^2). \label{eq:pm9-ACR2019}
	\end{align}
	Note that the quadratic form $F$ satisfies
	\begin{align*}
	& \ |F(a_1,a_2,\cdots) - F(b_1,b_2,\cdots)| \\
	\lesssim & \ |\sum_{i,j} (a_i - b_i + b_i) (a_j - b_j + b_j) - \sum_{i,j} b_i b_j| \\
	= & \ |\sum_{i,j} (a_i - b_i) (a_j - b_j) + 2\sum_{i,j} (a_i - b_i) b_j| \\
	\leq & \sum_{i} (a_i - b_i)^2 + 2 \big[ (1+M)\sum_{i} (a_i - b_i)^2 \big]^{1/2} \cdot \big[ \sum_{j} b_j^2/(1+M) \big]^{1/2} \\
	\leq & \ (2+M)\sum_{i} (a_i - b_i)^2 + \sum_{j} b_j^2/(1+M),
	\end{align*}
	for any positive number $M$.
	Summing up \eqref{eq:pm8-ACR2019} and \eqref{eq:pm9-ACR2019}, we arrive at
	\begin{align*}
	& \sum_\pm \int_{\Gamma_g} e^{2\sigma \varphi} (\calX (v_\pm / a_{-1,\pm})) |g|^{1/2} \dif S \nonumber\\
	\lesssim & \ \sigma \int_{\Gamma_g} e^{2\sigma\varphi} [(2+M)\calX ((v_+ / a_{-1,+}) - (v_- / a_{-1,-})) + \frac {\calX (v_- / a_{-1,-})} {1+M}] \dif S \\
	& + \sigma^2 e^{C\sigma} \sum_\pm (\nrm[H^1(\Sigma_{\pm};g')]{v_\pm}^2 + \nrm[L^2(\Sigma_{\pm};g')]{\partial_{\nu,g} v_\pm}^2).
	\end{align*}
	Let $\sigma$ be large enough so that all the arguments hold and then fix $M = \sigma^2$, thus the term $\sigma \int_{\Gamma_g} e^{2\sigma\varphi} \frac {\calX v_-} {1+M} \dif S$ can be absorbed by the left-hand-side and then it gives us
	\begin{align*}
	& \sum_\pm \int_{\Gamma_g} (\calX (v_\pm / a_{-1,\pm})) |g|^{1/2} \dif S \\
	\lesssim & \ \int_{\Gamma_g} \calX ((v_+ / a_{-1,+}) - (v_- / a_{-1,-})) \dif S
	+ \sum_\pm (\nrm[H^1(\Sigma_{\pm};g')]{v_\pm}^2 + \nrm[L^2(\Sigma_{\pm};g')]{\partial_{\nu,g} v_\pm}^2).
	\end{align*}
	The proof is complete.
\end{proof}

\section{Potential recovery with two measurements} \label{sec:ip2-ACR2019}

In this section, we consider the recovery of the potential $q$ utilizing scattering data generated by two incident plane waves coming from two opposite directions; 
we are devoted to the proof of Theorem \ref{thm:mst-ACR2019}.
Without loss of generality, we assume these two directions are $e_n$ and $-e_n$.
We denote these geometric-related quantities associated with $e_n$ and $-e_n$ as $\omega_i$, $\gamma_i$, $X_i^m$, $X_i$ and $\mathcal X_i$ for $i = 1,2$, and denote the two solutions coming from $e_n$ and $-e_n$ as
\begin{align*}
U(x,t) & = a_{-1}(x) \delta(t - \omega_1(x)) + u(x,t) H(t - \omega_1(x)), \\
V(x,t) & = a_{-1}'(x) \delta(t - \omega_2(x)) + v(x,t) H(t - \omega_2(x)),
\end{align*}
respectively.
Note that $a_{-1}$ and $a_{-1}'$ are not necessarily the same.
According to the assumption \eqref{assumption1-ACR2019}, we can assume that $\omega_1$ and $\omega_2$ satisfy  $\omega_1(x) = - \omega_2(x)$, and
\begin{equation} \label{eq:gamma12t-ACR2019}
\text{if~} \gamma_1(s) = (x,t), \text{~then~} \gamma_2(-s) = (x,-t).
\end{equation}
Moreover,
\begin{equation} \label{eq:X12m-ACR2019}
X^m_i = (\partial_j \omega_i) h^{jm} \partial_t + h^{mk} \partial_k
\quad \text{and} \quad
X_i = \partial_t + (\partial_j \omega_i) g^{jk}  \partial_k,
\quad i= 1,2.
\end{equation}

The following result extends \cite[Theorem 1.2]{RSfixed2019March} to the Riemannian case.

\begin{thm} \label{thm:XuvC2-ACR2019}
Assume $g$ satisfies the assumptions \eqref{assumption1-ACR2019} and \eqref{ass2-ACR2019}.
\sq{Let $T$ be as in Lemma \ref{lem:weight-ACR2019}.}
There exists a positive constant $C$ depending only on $T$, with the following property: if $q_1$, $q_2 \in C^\infty(\Rn;\R)$ are supported in $B$ and $u_1$, $u_2$, $v_1$ and $v_2$ are the corresponding solutions of \eqref{eq:1-ACR2019} and \eqref{eq:U-ACR2019}, then
\begin{equation} \label{eq:XuvC2-ACR2019}
\nrm[L^2(B)]{q_1 - q_2}^2 
\leq C(\nrm[H^1(\Gamma_g \cap \Sigma)]{u_1 - u_2}^2 + \nrm[H^1(\Sigma;g')]{u_1 - u_2}^2 + \nrm[H^1(\Sigma;g')]{v_1 - v_2}^2).
\end{equation}
\end{thm}

\begin{proof}[Proof of Theorem \ref{thm:mst-ACR2019}]
Theorem \ref{thm:XuvC2-ACR2019} provides a stability estimate.
The condition
\[
u_{q_1}^{\pm}|_{\partial B \times [-1,T]} = u_{q_2}^{\pm}|_{\partial B \times [-1,T]},
\]
implies $u_1 - u_2 = 0$ and $v_1 - v_2 = 0$ in \eqref{eq:XuvC2-ACR2019}, and hence $q_1 - q_2 = \tilde q = 0$.
\end{proof}

\begin{proof}[Proof of Theorem \ref{thm:XuvC2-ACR2019}]
Assume there are two potentials $q_1$ and $q_2$,
and we denote as $u_1$ and $v_1$ (resp.~$u_2$ and $v_2$) the solutions of \eqref{eq:1-ACR2019} and \eqref{eq:U-ACR2019} corresponding to $q_1$ (resp.~$q_2$).
Denote $\tilde q(x) = (q_1 - q_2)(x)$, $\tilde u(x,t) = (u_1 - u_2)(x,t)$ and $\tilde v(x,t) = -(v_1 - v_2)(x,-t)$, then it can be checked that
\begin{equation} \label{eq:gZv-ACR2019}
\square_g \tilde v(x,t) = -\big( \square_g (v_1 - v_2) \big) (x,-t),
\quad
X_1 \tilde v(y,z,z) = - \big( X_2 (v_1 - v_2) \big) (y,z,-z).
\end{equation}
where $x = (y,z)$ with $y \in \R^{n-1}$ and $z \in \R$.
Therefore,
\begin{equation} \label{eq:tbuv-ACR2019}
\left\{ \begin{aligned}
(\square_g + q_1) \tilde u(x,t) & = - \tilde q(x)\, u_2(x,t), && t > \omega_1(x), \\
(\square_g + q_1) \tilde v(x,t) & = \tilde q(x)\, v_2(x,-t), && t < -\omega_2(x) = \omega_1(x).
\end{aligned}\right.
\end{equation}
According to \eqref{eq:gamma12t-ACR2019} and the definition of $\tilde v$, we can conclude
\[
\tilde v(\gamma_1(s)) = -(v_1 - v_2)(\gamma_2(-s)).
\]
Combining this with \eqref{eq:uExp-ACR2019} and \eqref{eq:Xgasu-ACR2019}, we arrive at
\begin{equation} \label{eq:uvX1s-ACR2019}
\left\{ \begin{aligned}
\tilde u(\gamma_1(s)) 
& = -\frac 1 2 a_{-1}(\gamma_1(s)) \int_{-\infty}^s \tilde q(\gamma_1(\tau)) \dif \tau,
&
\text{so} \ \ 
X_1 |_{(x,\omega_1(x))} ( \tilde u / a_{-1}) 
& = - \tilde q(x) / 2, \\
\tilde v(\gamma_1(s)) 
& = \frac 1 2 a_{-1}' (\gamma_1(s)) \int_s^{+\infty} \tilde q(\gamma_1(\tau)) \dif \tau,
&
\text{so} \ \ 
X_1 |_{(x,\omega_1(x))} ( \tilde v / a_{-1}') 
& = - \tilde q(x) / 2.
\end{aligned}\right.
\end{equation}

By applying Proposition \ref{prop:vpmAsm-ACR2019} to \eqref{eq:tbuv-ACR2019} and \eqref{eq:uvX1s-ACR2019}, we obtain
\begin{align}
& \int_{\Gamma_g} [\calX_1 (\tilde u / a_{-1}) + \calX_1 (\tilde v / a_{-1}')] \, |g|^{1/2} \dif S \nonumber \\
\lesssim & \int_{\Gamma_g} \calX_1 (\tilde u / a_{-1} - \tilde v / a_{-1}') \dif S + \nrm[H^1(\Sigma_{+};g')]{\tilde u / a_{-1}}^2 + \nrm[H^1(\Sigma_{+};g')]{\tilde v / a_{-1}'}^2 \nonumber \\
& + \nrm[L^2(\Sigma_{+};g')]{\partial_{\nu,g} (\tilde u / a_{-1})}^2 + \nrm[L^2(\Sigma_{+};g')]{\partial_{\nu,g} (\tilde v / a_{-1}')}^2. \label{eq:XuvC-ACR2019}
\end{align}

Now we analyze these terms involving $\mathcal X_1$ in  \eqref{eq:XuvC-ACR2019}.
First, we analyze the left-hand-side of \eqref{eq:XuvC-ACR2019}.
For any function $\varphi \in C^1(\R_x^n \times \R_t)$, by \eqref{eq:Claim1Equ-ACR2019}, \eqref{eq:calX-ACR2019} and \eqref{eq:X12m-ACR2019},
we have
\begin{equation} \label{eq:X1g0u0-ACR2019}
\sqrt{|\nabla \omega_1|^2 + 1} (\calX_1 \varphi)
= \sigma^2 |\varphi|^2 + |X_1 \varphi|^2 + |\df \varphi - \df \omega_1 \agl[\df \omega_1, \df \varphi]|_g^2
\geq |X_1 \varphi|^2.
\end{equation}
Replacing $\varphi$ with $\tilde u / a_{-1}$ and with $\tilde v / a_{-1}'$ respectively in \eqref{eq:X1g0u0-ACR2019}, we obtain
\begin{align}
\sqrt{|\nabla \omega_1|^2 + 1} \calX_1 |_{(x,\omega_1(x))}(\tilde u / a_{-1})
& \geq \big| X_1 |_{(x,\omega_1(x))} (\tilde u / a_{-1}) \big|^2
= |\tilde q(x)|^2 / 4, \label{eq:X1ugtrq-ACR2019} \\
\sqrt{|\nabla \omega_1|^2 + 1} \calX_1 |_{(x,\omega_1(x))}(\tilde v / a_{-1}')
& \geq \big| X_1 |_{(x,\omega_1(x))} (\tilde v / a_{-1}') \big|^2
= |\tilde q(x)|^2 / 4. \label{eq:X1vgtrq-ACR2019}
\end{align}

Then we investigate the term $\calX_1 (\tilde u / a_{-1} - \tilde v / a_{-1}')$ on the right-hand-side of  \eqref{eq:XuvC-ACR2019}.
Define a mapping $F \colon \Gamma_g \to \Gamma_g \cap \Sigma$ in the following way.
For any $(x,\omega_1(x)) \in \Gamma_g$, there exists a unique \sq{$s > 0$} such that $\gamma_{(x,\omega_1(x))}(s) \in \Sigma$, and we define $F(x,\omega_1(x)) := \gamma_{(x,\omega_1(x))}(s)$.
It means that $F$ sends $(x,\omega_1(x))$ to the intersection between $\Sigma$ and the integral curve $\gamma$ of $X_1$ passing through the point $(x,\omega_1(x))$.
We denote as $F^a$ the $a$-th component ($a = 1,\cdots,n+1$) of $F$.
It can be checked that
\begin{subequations} \label{eq:Fc2-ACR2019}
\begin{numcases}{}
\sum_{a = 1}^n F^a(x,\omega_1(x))^2 = 1, \label{eq:Fc2.1-ACR2019} \\
F^n(x,\omega_1(x)) = F^{n+1}(x,\omega_1(x)). \label{eq:Fc2.2-ACR2019}
\end{numcases}
\end{subequations}

We show that $X_1 F = 0$. 
Note that $\{ \gamma_{(x,\omega_1(x))}(t) \,;\, t \in \R \} \subset \Gamma_g$.
According to the definition of $F$, for any fixed $(x,\omega_1(x))$, there exists an integral curve $\bar \gamma$ of $X_1$ and a number $s$ such that $\bar \gamma(0) = (x,\omega_1(x))$ and $\bar \gamma(s) = F(x,\omega_1(x))$, and moreover, $\bar \gamma(s) = F(\bar \gamma(t))$ for any $t$. 
Hence,
\[
X_1 F(x,\omega_1(x))
= X_1 |_{(x,\omega_1(x))} F
= \frac{\df}{\df t} \big( F(\bar \gamma(t)) \big) |_{t = 0}
= \frac{\df}{\df t} \big( \bar \gamma(s) \big) |_{t = 0}
= 0.
\]
From $X_1 F = 0$ and \eqref{eq:XDef-ACR2019}, we can obtain
\begin{equation} \label{eq:Fc1-ACR2019}
\partial_t F^a(x,\omega_1(x)) + \partial_j \omega_1 g^{jk} \partial_k F^a(x,\omega_1(x)) = 0.
\end{equation}

Keeping in mind that the potential is assumed to be supported in $B(0,1)$, 
then formula \eqref{eq:uvX1s-ACR2019} gives
\begin{equation} \label{eq:uvu-ACR2019}
\frac {\tilde u}  {a_{-1}} (x, \omega_1(x)) - \frac {\tilde v}  {a_{-1}'} (x, \omega_1(x))
= -\frac 1 2 \int_{-\infty}^{+\infty} \tilde q(\gamma_1(\tau)) \dif \tau 
= \frac {\tilde u} {a_{-1}} (F(x, \omega_1(x))). 
\end{equation}
Recall that $X_1^m = (\partial_j \omega_1) h^{jm} \partial_t + h^{mk} \partial_k$, so for any $\varphi \in C^1(\R_x^n \times \R_t)$, \sq{it} can checked that
\begin{equation} \label{eq:XmInOut-ACR2019}
X_1^m |_{(x,\omega_1(x))} \varphi 
= X_1^m \varphi(x,\omega_1(x))
= h^{mk} \partial_k \big( \varphi(x,\omega_1(x)) \big).
\end{equation}
Therefore, by combining \eqref{eq:Fc1-ACR2019}, \eqref{eq:uvu-ACR2019} and \eqref{eq:XmInOut-ACR2019}, one can compute
\begin{align}
& \ \quad X_1^m |_{(x,\omega_1(x))} ( \tilde u / a_{-1} - \tilde v / a_{-1}' )
= h^{mk} \partial_k \Big( \frac {\tilde u}  {a_{-1}} \circ F (x,\omega_1(x)) \Big) \nonumber \\
& = h^{mk} \sum_{a = 1}^{n+1} \partial_a \big( \tilde u / a_{-1} \big)(F(x,\omega_1(x))) \cdot \big( \partial_k F^a + \partial_t F^a (\partial_k \omega_1) \big) \nonumber \\
& = \ \sum_{a = 1}^{n+1} h^{mk} \big( \partial_k F^a(x,\omega_1(x)) - (\partial_j \omega_1) g^{j\ell} (\partial_\ell F^a(x,\omega_1(x))) (\partial_k \omega_1) \big) \partial_a |_{F(x,\omega_1(x))} ( \tilde u / a_{-1} ) \nonumber \\
& =: Y^m |_{F(x,\omega_1(x))} ( \tilde u / a_{-1} ), \label{eq:Xmua-ACR2019}
\end{align}
where $Y^m |_{F(x,\omega_1(x))}$ is the value at $F(x,\omega_1(x))$ of the vector field \sq{$Y^m$}.

Combining \eqref{eq:calX-ACR2019}, \eqref{eq:uvu-ACR2019} and  \eqref{eq:Xmua-ACR2019}, we arrive at
\begin{align}
& \ \sqrt{|\nabla \omega_1|^2 + 1} \calX_1 |_{(x,\omega_1(x))} ( \tilde u / a_{-1} - \tilde v / a_{-1}' ) \nonumber \\
= & \ \sigma^2 \big| \frac {\tilde u} {a_{-1}} (F(x, \omega_1(x))) \big|^2 + \sum_{m = 1}^{n} \big| \ Y^m |_{F(x,\omega_1(x))} ( \tilde u / a_{-1} ) \big|^2. \label{eq:X1g0.2-ACR2019}
\end{align}

\sq{Now we make the following claim.}
\begin{equation} \label{eq:XmCla-ACR2019}
\textbf{Claim:~} Y^m |_{F(x,\omega_1(x))}
\text{~is tangential to~} \Gamma_g \cap \Sigma,
\end{equation}
and this can be justified as follows.
Denoting $H_1(x,t) = t - \omega_1(x)$ and $H_2(y,z,t) = |x|^2/2 - 1/2$, it is clear that $\Gamma_g \cap \Sigma$ is the intersection of $\{(x,t) \,;\, H_1(x,t) = 0\}$ and $\{(x,t) \,;\, H_2(x,t) = 0\}$.
First, formula \eqref{eq:Fc2.1-ACR2019} implies that the $n$-dimensional vector $\big( F^a(x,\omega_1(x)) \big)_{a = 1}^n$ belongs to $\partial B$, so by \eqref{assumption1-ACR2019} we have
\[
\partial_a |_{F^a(x,\omega_1(x))} H_1 = \partial_a |_{F^a(x,\omega_1(x))} (t - z) = 
\left\{ \begin{aligned}
& 0, && a < n, \\
& -1, && a = n, \\
& 1, && a = n + 1.
\end{aligned} \right.
\]
Therefore,
\begin{align}
Y^m |_{F(x,\omega_1(x))} H_1
& = \big[ h^{mk} \big( \partial_k F^{n+1}(x,\omega_1(x)) - (\partial_j \omega_1) g^{j\ell} (\partial_\ell F^{n+1}(x,\omega_1(x))) (\partial_k \omega_1) \big) \big] \nonumber \\
& \quad - \big[ h^{mk} \big( \partial_k F^n(x,\omega_1(x)) - (\partial_j \omega_1) g^{j\ell} (\partial_\ell F^n(x,\omega_1(x))) (\partial_k \omega_1) \big) \big] \nonumber \\
& = 0. \label{eq:H10-ACR2019}
\end{align}
Note that in \eqref{eq:H10-ACR2019} we used \eqref{eq:Fc2.2-ACR2019}.
Second, we have
\[
\partial_a |_{F^a(x,\omega_1(x))} H_2 = \partial_a |_{F^a(x,\omega_1(x))} (|x|^2/2 - 1/2) = 
\left\{ \begin{aligned}
& F^a(x,\omega_1(x)), && a \leq n, \\
& 0, && a = n + 1.
\end{aligned} \right.
\]
Hence $Y^m |_{F(x,\omega_1(x))} H_2$ can be computed as follows,
\begin{align}
Y^m |_{F(x,\omega_1(x))} H_2
& = \sum_{a = 1}^n h^{mk} \big[ \partial_k F^a(x,\omega_1(x)) - (\partial_j \omega_1) g^{j\ell} (\partial_\ell F^a(x,\omega_1(x))) (\partial_k \omega_1) \big] F^a(x,\omega(x)) \nonumber \\
& = \frac 1 2 h^{mk} \big[ \partial_k \big( \sum_{a = 1}^n F^a(x,\omega_1(x))^2 \big) - (\partial_j \omega_1) g^{j\ell} \partial_\ell \big( \sum_{a = 1}^n F^a(x,\omega_1(x))^2 \big) (\partial_k \omega_1) \big] \nonumber \\
& = \frac 1 2 h^{mk} \big[ \partial_k (1) - (\partial_j \omega_1) g^{j\ell} \partial_\ell (1) (\partial_k \omega_1) \big] 
= 0. \label{eq:H20-ACR2019}
\end{align}
Note that in \eqref{eq:H20-ACR2019} we used \eqref{eq:Fc2.1-ACR2019}.
Formulae \eqref{eq:H10-ACR2019} and \eqref{eq:H20-ACR2019} imply that $Y_m |_{F(w,\omega_1(x))}$ is tangential to both $\{ H_1 = 0\}$ and $\{ H_2 = 0\}$, and hence tangential to $\{ H_1 = 0\} \cap \{ H_2 = 0\}$ which is $\Gamma_g \cap \Sigma$.

Combining the claim \eqref{eq:XmCla-ACR2019} with \eqref{eq:X1g0.2-ACR2019}, 
we can obtain
\begin{align}
\int_{\Gamma_g}\calX_1 ( \tilde u / a_{-1} - \tilde v / a_{-1}' ) \dif S
& \lesssim \int_{\Gamma_g} \sigma^2 \big| \frac {\tilde u} {a_{-1}} (F(x, \omega_1(x))) \big|^2 + \sum_{m = 1}^{n} \big| \ Y^m |_{F(x,\omega_1(x))} ( \tilde u / a_{-1} ) \big|^2 \dif S \nonumber \\
& \lesssim \sigma^2 \nrm[L^2(\Gamma_g \cap \Sigma)]{\tilde u / a_{-1}}^2 + \nrm[H^1(\Gamma_g \cap \Sigma)]{\tilde u / a_{-1}}^2. \label{eq:X1nromuv-ACR2019}
\end{align}
Now we combine \eqref{eq:XuvC-ACR2019}, \eqref{eq:X1ugtrq-ACR2019}, \eqref{eq:X1vgtrq-ACR2019} and \eqref{eq:X1nromuv-ACR2019}, and we obtain
\begin{align*}
& \ \nrm[L^2(B(0,1))]{\tilde q}^2 \nonumber \\
\lesssim &  \int_{\Gamma_g} [\calX_1 (\tilde u / a_{-1}) + \calX_1 (\tilde v / a_{-1}')] \, |g|^{1/2} \dif S \nonumber \\
\lesssim & \ \sigma^2 \nrm[L^2(\Gamma_g \cap \Sigma)]{\tilde u / a_{-1}}^2 + \nrm[H^1(\Gamma_g \cap \Sigma)]{\tilde u / a_{-1}}^2 + \nrm[H^1(\Sigma_{+};g')]{\tilde u / a_{-1}}^2 + \nrm[H^1(\Sigma_{-};g')]{\tilde v / a_{-1}'}^2 \nonumber \\
& + \nrm[L^2(\Sigma_{+};g')]{\partial_{\nu,g} (\tilde u / a_{-1})}^2 + \nrm[L^2(\Sigma_{-};g')]{\partial_{\nu,g} (\tilde v / a_{-1}')}^2 \nonumber \\
\lesssim & \ \sigma^2 \nrm[L^2(\Gamma_g \cap \Sigma)]{\tilde u}^2 + \nrm[H^1(\Gamma_g \cap \Sigma)]{\tilde u}^2 + \nrm[H^1(\Sigma_{+};g')]{\tilde u}^2 + \nrm[H^1(\Sigma_{-};g')]{\tilde v}^2 \nonumber \\
& + \nrm[L^2(\Sigma_{+};g')]{\partial_{\nu,g} \tilde u}^2 + \nrm[L^2(\Sigma_{-};g')]{\partial_{\nu,g} \tilde v}^2.
\end{align*}
Applying \cite[Lemma 3.3]{RSfixed2019March}, we can conclude the proof.
\end{proof}

\appendix

\section{Equivalence of conditions} \label{sec_equivalence}

We prove the equivalence of three conditions appearing in the introduction. The required arguments are standard (one could also use the language of Lagrangian manifolds as in \cite[Section 6.4]{Hormander1}). We give the details for completeness.

\begin{lem} \label{lemma_omega_equivalence}
Let $g$ be a smooth Riemannian metric in $\R^n$ satisfying \eqref{asm:gce-ACR2019}. Then the conditions \eqref{assumption1-ACR2019}, \eqref{assumption1a-ACR2019}, and \eqref{assumption1b-ACR2019} are equivalent.
\end{lem}
\begin{proof}
We first prove that \eqref{assumption1-ACR2019} implies \eqref{assumption1a-ACR2019}. First extend $\omega$ smoothly to $\R^n$ so that $\omega(x) = x_n$ outside $B$. If $\gamma$ is any unit speed $g$-geodesic, we claim that 
\begin{equation} \label{omega_bicharacteristic_conditions}
d\omega(\gamma(t_0)) = \dot{\gamma}(t_0)^{\flat} \text{ for some $t_0$} \implies d \omega(\gamma(t)) = \dot{\gamma}(t)^{\flat} \text{ for all $t \in \R$}.
\end{equation}
To prove this, write $p(x,\xi) = \frac{1}{2}(\abs{\xi}_g^2 - 1)$ where $(x,\xi) \in T^* \R^n = \R^n \times \R^n$, let $X(t)$ solve $\dot{X}(t) = \nabla_{\xi} p(X(t), \nabla \omega(X(t)))$ with $X(t_0) = \gamma(t_0)$ and write $\Xi(t) = \nabla \omega(X(t))$. We compute 
\[
\dot{\Xi}_j(t) = \partial_{jk} \omega(X(t)) \dot{X}_k(t) = \partial_{jk} \omega(X(t)) \p_{\xi_k} p (X(t), \nabla \omega(X(t))).
\]
But \eqref{assumption1-ACR2019} gives that $p(x, \nabla \omega(x)) = 0$, and differentiating this equation yields 
\[
\p_{x_j} p(x, \nabla \omega(x)) + \p_{\xi_k} p(x, \nabla \omega(x)) \p_{jk} \omega(x) = 0.
\]
Combining the last two equations shows that $(X(t), \Xi(t))$ satisfies the Hamilton equations
\[
\dot{X}(t) = \nabla_{\xi} p(X(t), \Xi(t)), \qquad \dot{\Xi}(t) = -\nabla_x p(X(t), \Xi(t)).
\]
Note also that $(X(t_0),\Xi(t_0)) = (\gamma(t_0), \dot{\gamma}(t_0)^{\flat})$, using that $d\omega(\gamma(t_0)) = \dot{\gamma}(t_0)^{\flat}$. Since solutions are unique, $(X(t), \Xi(t))$ must agree with the null bicharacteristic $(\gamma(t), \dot{\gamma}(t)^{\flat})$ for $p$. This shows \eqref{omega_bicharacteristic_conditions}. By integrating, we also see that 
\begin{equation} \label{omega_bicharacteristic_conditions_second}
\omega(\gamma(t_2)) - \omega(\gamma(t_1)) = \int_{t_1}^{t_2} d\omega |_{\gamma(t)} (\dot{\gamma}(t)) \dif t = \int_{t_1}^{t_2} \agl[\df \omega, \df \omega] \dif t = t_2 - t_1.
\end{equation}

For any $x' \in \R^{n-1}$ let $\gamma_{x'}(t)$ be the unit speed $g$-geodesic with $\gamma_{x'}(-1) = (x',-1)$ and $\dot{\gamma}_{x'}(-1) = e_n$. Then $d\omega(\gamma_{x'}(-1)) = \dot{\gamma}_{x'}(-1)^{\flat}$ since $\omega = x_n$ outside $B$, so \eqref{omega_bicharacteristic_conditions} gives that $d \omega(\gamma_{x'}(t)) = \dot{\gamma}_{x'}(t)^{\flat}$ for all $t$. Consider the map 
\[
F: \R^n \to \R^n, \ \ F(x',t) = \gamma_{x'}(t).
\]
This map is smooth, and it is the identity map in $\R^n \setminus \{ (x',t) \,:\, \abs{x'} \leq 1, \,t \geq -1 \}$. It is injective, since if $\gamma_{x'}(t) = \gamma_{\tilde{x}'}(\tilde{t})$ then by \eqref{omega_bicharacteristic_conditions} 
\[
\dot{\gamma}_{x'}(t)^{\flat} = d\omega(\gamma_{x'}(t)) = d\omega(\gamma_{\tilde{x}'}(\tilde{t})) = \dot{\gamma}_{\tilde{x}'}(\tilde{t})^{\flat}.
\]
Thus the two geodesics must be the same, i.e.~$\gamma_{x'}(t+s) = \gamma_{\tilde{x}'}(\tilde{t}+s)$. 
By taking $s \ll -1$ we obtain $x' = \tilde{x}'$, and \eqref{omega_bicharacteristic_conditions_second} then gives $t = \tilde{t}$. 
Finally, we show that $F$ is surjective. 
For any $y$ with $\abs{y'} \leq 1$ and $y_n > -1$, let $\gamma$ be the geodesic with $\gamma(0) = y$ and $\dot{\gamma}(0)^{\flat} = d\omega(y)$.
We denote $A: = \inf_{\{x \in \Rn ; x_n > -1\}} \omega(x)$ for short, and it can be checked that $A$ is finite.
If $\gamma_n(s) > -1$ for all $s \leq 0$, then $\omega(\gamma(s)) \geq A$ for all $s \leq 0$ (using that $\omega$ is smooth in $\ol{B}$ and $\omega = x_n$ outside $B$). 
Now \eqref{omega_bicharacteristic_conditions_second} implies that 
\[
s = \omega(\gamma(s)) - \omega(\gamma(0)) \geq A - \omega(\gamma(0)) \text{ for all $s \leq 0$}. 
\]
This is a contradiction, which proves that $y = \gamma_{x'}(t)$ for some $x'$ and $t$. A similar argument for $y = (x',-1)$ shows that for any $x' \in \R^{n-1}$, the geodesic $\gamma_{x'}(t)$ meets $\{ x_n = 1 \}$.

We have proved that $F$ is smooth and bijective, i.e.\ a diffeomorphism. Hence the map $(x',t) \mapsto \gamma_{x'}(t)$ smoothly parametrizes $\R^n$, and any geodesic starting from $\{ x_n = -1 \}$ in direction $e_n$ meets $\{ x_n = 1 \}$ pointing in direction $e_n$. This concludes the proof that \eqref{assumption1-ACR2019} implies \eqref{assumption1a-ACR2019}.

Now, if \eqref{assumption1a-ACR2019} is satisfied, then the map $F(x',t) = \gamma_{x'}(t)$ above is a diffeomorphism of $\mR^n$. Writing $y' = x'$ and $y_n = t$ gives a global coordinate system $y = (y',y_n)$ in $\R^n$. (These are sometimes called semigeodesic coordinates with respect to $\{ x_n = -1 \}$.) In particular, $y_n(x) = t$ if $x = \gamma_{z'}(t)$ for some (unique) $z'$ and $t$, and then $\partial_{y_n}|_x = \dot{\gamma}_{z'}(t)$. The fact that the metric is Euclidean outside $B$ together with \eqref{assumption1a-ACR2019} ensure that $\partial_{y_n} = \partial_{x_n}$ outside $B$. Clearly we have $\abs{\partial_{y_n}}_g = 1$, so the metric has the form 
\[
g(y) = \left( \begin{array}{cc} g_0(y) & * \\ * & 1 \end{array} \right).
\]
It remains to show that $g(\partial_{y_{\alpha}}, \partial_{y_n}) = 0$ for $1 \leq \alpha \leq n-1$. This is a standard property of semigeodesic coordinates, and it follows from the computation 
\[
\p_{y_n} (g(\partial_{y_{\alpha}}, \partial_{y_n})) = g(\nabla_{\partial_{y_n}} \partial_{y_{\alpha}}, \partial_{y_n}) + g(\partial_{y_{\alpha}}, \nabla_{\partial_{y_n}} \partial_{y_n})
\]
where $\nabla$ is the Levi-Civita connection for $g$. The last term vanishes since $\partial_{y_n}$ is the tangent vector of a geodesic. For the first term on the right we use that $\nabla$ is torsion free, i.e.\ $\nabla_{\partial_{y_n}} \partial_{y_{\alpha}} = \nabla_{\partial_{y_{\alpha}}} \partial_{y_n}$, and this implies 
\[
\p_{y_n} g(\partial_{y_{\alpha}}, \partial_{y_n}) = \frac{1}{2} \p_{y_{\alpha}} ( g(\partial_{y_n}, \partial_{y_n}) ) = 0.
\]
Since $g(\partial_{y_{\alpha}}, \partial_{y_n})|_{(x',-1)} = e(\partial_{x_{\alpha}}, \partial_{x_n}) = 0$ where $e$ is the Euclidean metric, we see that $g(\partial_{y_{\alpha}}, \partial_{y_n}) = 0$ everywhere. This shows \eqref{assumption1b-ACR2019}.

Finally, if \eqref{assumption1b-ACR2019} holds then we may choose $\omega(y) = y_n$. The function $\omega$ is smooth everywhere, and the form of the metric implies that $\abs{d\omega}_g = 1$. We also have $d\omega = dx_n$ outside $B$, hence $\omega = x_n + C$ outside $B$ for some constant $C$. Subtracting the constant gives the function $\omega$ required in \eqref{assumption1-ACR2019}.
\end{proof}


{

}

\end{document}